\documentclass[12pt,reqno]{amsart}
\usepackage{amssymb}
\usepackage{amsmath}
\usepackage{amsfonts}
\usepackage{amsthm}
\usepackage{enumitem}
\usepackage{graphicx}
\usepackage{extarrows}
\usepackage{xcolor}
\usepackage{bm}
\usepackage{mathtools}
\usepackage{caption}
\usepackage{subcaption}
\usepackage{hyperref}
\usepackage{diagbox}
\usepackage[noadjust]{cite}
\usepackage[margin=1in]{geometry}
\newtheorem{thm}{Theorem}[section]
\newtheorem{remark}[thm]{Remark}
\newtheorem{pro}[thm]{Proposition}

\newtheorem{exam}[thm]{Example}
\newtheorem{defn}[thm]{Definition}
\newtheorem{lemma}[thm]{Lemma}

\newtheorem{que}[thm]{Question}

\numberwithin{equation}{section}

\usepackage{xcolor}
\begin{document}

\title{On the Spectral Properties of a Class of Planar Sierpinski Self-Affine Measures}
\author{Jia-Long Chen$^{1}$}
\author{Wen-Hui Ai$^{~2\mathbf{*}}$}
	
\address{$^1$ School of Mathematics, South China University of Technology, Guangzhou, 510641, P. R. China}
\address{$^2$ Key Laboratory of Computing and Stochastic Mathematics (Ministry of Education),
School of Mathematics and Statistics, Hunan Normal University, Changsha, Hunan 410081, P. R. China}
\email{jialongchen1@163.com}
\email{awhxyz123@163.com}
	\date{\today}
	\keywords{self-affine measure, spectral measure, orthonormal basis, non-spectral.}
	\subjclass[2010]{Primary 28A80; Secondary 42C05, 46C05}
\thanks{The research is supported in part by the NNSF of China (Nos. 12201206 and 12371072), the Hunan Provincial NSF (No. 2024JJ6301).\\
$^*$Corresponding author.}
	
	\begin{abstract}
		We investigate the spectral properties of a class of Sierpinski-type self-affine measures defined by
		\[
		\mu_{M,D}(\cdot) = p^{-1} \sum_{d \in D} \mu_{M,D}(M(\cdot) - d),
		\]
		where \( p \) is a prime number, \( M = \begin{bmatrix}
			\rho_1^{-1} & c \\
			0 & \rho_2^{-1}
		\end{bmatrix} \) is a real upper triangular expanding matrix, and \( D = \{d_0, d_1, \cdots, d_{p-1}\} \subset \mathbb{Z}^2 \) satisfying \( \mathcal{Z}(\widehat{\delta}_{D}) = \cup_{j=1}^{p-1} \left( \frac{j \bm{a}}{p} + \mathbb{Z}^2 \right) \) for some \( \bm{a} \in \mathcal{E}_{p}= \{ (i_1, i_2)^* : i_1,  i_2 \in [1, p-1] \cap \mathbb{Z} \} \), where \( \mathcal{Z}(\widehat{\delta}_{D}) \) denotes the set of zeros of \( \widehat{\delta}_{D} \) with \( \delta_{D} = \frac{1}{\# D} \sum_{d \in D} \delta_d \). When $\rho_1 = \rho_2$, we derive necessary and sufficient conditions for $\mu_{M,D}$ to both: $(i)$ possess an infinite orthogonal set of exponential functions, and $(ii)$ be a spectral measure. When no infinite orthogonal exponential system exists in $L^{2}(\mu_{M,D})$, we quantify the maximum number of orthogonal exponentials and provide precise estimates. For $\rho_1 \neq \rho_2$, with restricted digit sets $D$, we obtain a necessary and sufficient condition for $\mu_{M,D}$ to be a spectral measure.
	\end{abstract}

	\maketitle

\section{Introduction}

	Let \( \mu \) be a Borel probability measure on \( \mathbb{R}^n \) with compact support, and let \( L^2(\mu) \) denote the Hilbert space associated with \( \mu \). The problem of approximating functions in \( L^2(\mu) \) by ``well-behaved" functions has a long-standing history. This area connects various fields such as analysis, geometry, and topology. The most effective approximation occurs when \( L^2(\mu) \) possesses a basis consisting of complex exponentials, forming what is known as the Fourier basis. This work investigates the existence conditions for Fourier bases in  \( L^2(\mu) \), with special attention to singular measures $\mu$ on \( \mathbb{R}^n \). More precisely, we say that \( L^2(\mu) \) has a Fourier basis if there exists a countable subset \( \Lambda \subset \mathbb{R}^n \) such that the set of complex exponentials
	\[
	E_\Lambda = \left\{ e^{-2\pi i \langle \lambda, x \rangle} : \lambda \in \Lambda \right\}
	\]
	forms an orthonormal basis for \( L^2(\mu) \). In this case, we refer to \( \mu \) as a spectral measure, \( \Lambda \) as the spectrum of \( \mu \), and \( (\mu, \Lambda) \) as a spectral pair. Specifically, if the normalized Lebesgue measure restricted to a set \( \Omega \) is a spectral measure, we refer to \( \Omega \) as a spectral set.
	
In 1974, Fuglede established a significant connection between the existence of commuting self-adjoint partial differential operators and spectrality, leading to the formulation of the renowned spectral set conjecture in \cite{B1}. This conjecture is known to be false in dimensions \( n \geq 3 \)\cite{T1,KM1, KM2}. However interest in this conjecture is still alive.
	
Let \(\{\varphi_d(x)\}_{d \in D}\) be an iterated function system (IFS) defined by
	\[
	\varphi_d(x) = M^{-1}(x + d), \quad x \in \mathbb{R}^n, \, d \in D,
	\]
where \(M \in M_n(\mathbb{R})\) is an \(n \times n\) expanding real matrix (i.e., all eigenvalues of \(M\) have absolute values greater than 1), and \(D \subset \mathbb{R}^n\) is a finite set of digits. It is well known that there exists a unique non-empty compact set
	\[
	T :=T(M, D) = \left\{ \sum_{k=1}^{\infty} M^{-k} d_k : d_k \in D \right\}= \sum_{k=1}^{\infty} M^{-k} D
	\]
	such that $T = \cup_{d \in D} \varphi_d(T) $ \cite{HJ}.
	Additionally, there exists a unique probability measure $\mu := \mu_{M,D}$ supported on $T$ that satisfies
	\begin{equation}\label{1.1}
		\mu = \frac{1}{\#D} \sum_{d \in D} \mu \circ \varphi_d^{-1},
	\end{equation}
	where $\#D$ denotes the cardinality of $D$. The set \( T \) and the measure \( \mu_{M,D} \) are called the self-affine set (or attractor) and the self-affine measure, respectively. In particular, if $M$ is a scalar multiple of an orthonormal matrix, then $T$ and $\mu_{M,D}$ are called the self-similar set and self-similar measure, respectively.

In 1998, Jorgensen and Pedersen \cite{JP1} discovered the first singular and non-atomic spectral measure. They proved that the one fourth Cantor measure $\mu_{4}$ supported on $T=\{\sum_{i=1}^{\infty}4^{-i}d_{i}: \;d_{i}\in\{0,2\}\}$, is a spectral measure.
Since then, the study of singular spectral measures has flourished in the field of fractal analysis. For the spectrality of Bernoulli convolution in $\mathbb R$, it was fully resolved, see \cite{HL,Dai1,D2,DHL14}.
In \(\mathbb{R}^2\), the most work of the spectra measure was concentrated on the self-affine Sierpinski-type measures, see \cite{DL15,DFY,LDL}. In \cite{LDL},  Lu et al. investigated the self-affine measure generated by the upper triangular matrices and some ternary digit sets.
Recently, Yan \cite{Y1} extends the results of \cite{DFY} to the Moran measure generated by an expanding real diagonal matrix and some sets of $p$-ary numbers satisfying a certain zero condition, where $p$ is a prime. Fractal measures associated with this zero set have been extensively studied, yielding many significant results; see \cite{CYZ,CC,CLZ,WLS1} for further details.

Their results have inspired us to investigate the spectral properties of the self-affine measure
\( \mu_{M,D} \) with
\begin{equation}\label{1.2}
	M = \begin{bmatrix}
		\rho_1^{-1} & c \\
		0 & \rho_2^{-1}
	\end{bmatrix} \subset M_{2}(\mathbb{R}),
\quad D = \left\{
	d_0, d_1, \cdots, d_{p-1}
	\right\} \subset \mathbb{Z}^{2},
\end{equation}
where \( 0 < \rho_1, \rho_2< 1 \) and \( p \) is a prime number. Let $\widehat{\mu}$ be the Fourier transform of $\mu$, and \(\mathcal{Z}(f)\) the zeros of \(f\). The technique is to make use
of some explicit expressions of \( \mathcal{Z}(\widehat{\mu}_{M,D})\).
In this paper, we consider the set $D$ satisfying
\begin{equation}\label{x1.3}
		\mathcal{Z}(\widehat{\delta}_{D}) = \bigcup_{j=1}^{p-1} \left( \frac{j \bm{a}}{p} + \mathbb{Z}^2 \right), \bm{a} \in \mathcal{E}_{p}= \{ (i_1,  i_2)^* : i_1, i_2 \in [1, p-1] \cap \mathbb{Z} \}.
\end{equation}
Here $(i_1,  i_2)^*$ denotes the transposed conjugate of $(i_1,  i_2)$, $\delta_{D} = \frac{1}{\# D} \sum_{d \in D} \delta_d $ and $\delta_d $
represents the Dirac measure at \(d\).
Actually, many digit sets satisfy the form in equation \eqref{x1.3}. For example, let \( D_1 = \{0, 1, \cdots, N-1\} \) and \( D_2 = \{(0,0)^*, (1,0)^*, (0,1)^*\} \), by direct calculation, we have
	\( \mathcal{Z}(\widehat{\delta}_{D_{1}}) = \{ \frac{j}{N} + \mathbb{Z} : j = 1, 2, \cdots, N-1 \} \), \( \mathcal{Z}(\widehat{\delta}_{D_{2}}) = \{ \pm ( \frac{1}{3}, -\frac{1}{3} )^* + \mathbb{Z}^2 \} \).

Suppose that \( \rho_1= \rho_2\). We first establish the following two results, which provide necessary and sufficient conditions for \( L^{2}(\mu_{M, D}) \) to admit an infinite set of orthogonal exponential functions. If \( L^{2}(\mu_{M, D}) \) does not possess an infinite orthogonal set, we can estimate the maximal number of orthogonal exponential functions and determine its exact value.
	\begin{thm}\label{thm1.7}
		 Suppose that \( \rho_1= \rho_2= \rho\). Let \(\mu_{M,D}\) be given by \eqref{1.1}, \eqref{1.2}and \eqref{x1.3}. Then \(L^2(\mu_{M,D})\) admits an infinite orthogonal set of exponential functions if and only if \(\rho^{-1} = (\frac{t}{s})^{\frac{1}{r}}\) and $c=\kappa\rho^{-1}$ for some \(s, t, r \in \mathbb{N}\) with \(t \in p\mathbb{Z}\), \(\gcd(s, t) = 1\) and \(\kappa \in \mathbb{Q}\).
	\end{thm}
	\begin{thm}\label{thm1.8}
		Suppose that \( \rho_1= \rho_2= \rho\). Let \(\mu_{M,D}\) be given by \eqref{1.1} \eqref{1.2}and \eqref{x1.3}. If \( \rho^{-1}= (\frac{t}{s})^{\frac{1}{r}}\) for some \(s, t, r \in \mathbb{N}\) with \(\gcd(s, t) = 1\). Then the following statements hold:
		\begin{enumerate}
			\item[$(i)$] If \(c \neq \kappa \rho^{-1}\) for any \(\kappa \in \mathbb{Q}\), then there exist at most $p$ mutually orthogonal exponential functions in \(L^2(\mu_{M,D})\), and the number $p$ is the best.
			
			\item[$(ii)$] If \(c = \kappa \rho^{-1}\) for some \(\kappa \in \mathbb{Q}\) and \(t \notin p\mathbb{Z}\), then
			\begin{enumerate}
				\item[$(a)$] If \(s \notin p\mathbb{Z}\), then there exist at most $p$ mutually orthogonal exponential functions in \(L^2(\mu_{M,D})\), and the number $p$ is the best.
				
				\item[$(b)$] If \(s \in p\mathbb{Z}\), then there are any number of orthogonal exponential functions in \(L^2(\mu_{M,D})\).
			\end{enumerate}
		\end{enumerate}
	\end{thm}

The following result establishes a necessary and sufficient condition for \(\mu_{M, D}\) to be a spectral measure when \(\rho_1 = \rho_2\). Consequently, \cite[Conjecture 4.2]{CWZ} is incorrect, see the counterexample \ref{ex6.1}.

\begin{thm}\label{thm1.5}
	Suppose that \( \rho_1= \rho_2= \rho\). Let \(\mu_{M,D}\) be given by \eqref{1.1} \eqref{1.2}and \eqref{x1.3}. Then, \( \mu_{M,D} \) is a spectral measure if and only if \( \rho^{-1} \in p\mathbb{Z} \) and \( c \in  \{ \frac{t}{s} : t \in p\mathbb{Z}, \gcd(t, s) = 1 \}\cup\{ 0 \} \).
\end{thm}
For the case \(\rho_1 \neq \rho_2\), let \(d_0 = \bm{0}\), \(d_1 = (d_{1,1}, 0)^*\neq\bm{0}\), then
 \begin{equation}\label{eq1.3}
 D = \{d_0, d_1, \cdots, d_{p-1}\}= \left\{
 \begin{pmatrix} 0 \\ 0 \end{pmatrix},
 \begin{pmatrix} d_{1,1} \\ 0 \end{pmatrix},
 \begin{pmatrix} d_{2,1} \\ d_{2,2} \end{pmatrix},\cdots,
 \begin{pmatrix} d_{p-1,1} \\ d_{p-1,2} \end{pmatrix}
 \right\} \subset \mathbb{Z}^2.
 \end{equation}
Define
 \[
 R := \begin{bmatrix} \frac{1}{d_{1,1}} & \frac{c}{d_{1,1}(\rho_1^{-1} - \rho_2^{-1})} \\ 0 & 1 \end{bmatrix},
 \]
 and
 \[
 M'=RMR^{-1} = \begin{bmatrix} \rho_1^{-1} & 0 \\ 0 & \rho_2^{-1} \end{bmatrix}, \ \ \ D' = RD.
 \]
 It is easy to check that \( \mu_{M,D} \) and \( \mu_{M', D'} \) have the same spectrality. Let \( c' = \frac{c}{d_{1,1}(\rho_1^{-1} - \rho_2^{-1})} \), \( c'' = d_{1,1} c' \). For
 $\bm{a}= (a_1, a_2)^*  \in \mathcal{E}_{p} = \{ (i_1,  i_2)^* : i_1, i_2 \in [1, p-1] \cap \mathbb{Z}\}$, define
 \begin{equation}\label{1.3}
  E_{\bm{a}} := \left\{ \frac{h}{l} : la_2 - ha_1 \in \mathbb{Z} \setminus p\mathbb{Z}, \gcd(h, l) = 1 \right\} \cup \{0\}.
 \end{equation}
 With this definition, we establish the following result.
\begin{thm}\label{thm1.2}
 	Suppose that \( \rho_1 \neq \rho_2 \). Let \(\mu_{M,D}\) be given by \eqref{1.1}, \eqref{1.2}, \eqref{x1.3} and \eqref{eq1.3}. If \( c'' \notin \mathbb{Q} \), then \( \mu_{M,D} \) is not a spectral measure.
\end{thm}
This result also indicates that if \(\rho_1 \neq \rho_2\), a necessary condition for \(\mu_{M, D}\) to be a spectral measure is \(c'' \in \mathbb{Q}\). Consequently, when examining the necessary conditions, we can restrict our attention to the case where \(c'' \in \mathbb{Q}\). Divide \( \mathbb{Q} \) into \( \mathbb{Q} \setminus E_{\bm{a}} \) and \( E_{\bm{a}} \), yielding the following two results.

\begin{thm}\label{thm1.3}
	Suppose that \( \rho_1 \neq \rho_2 \). Let \(\mu_{M,D}\) be given by \eqref{1.1}, \eqref{1.2}, \eqref{x1.3} and \eqref{eq1.3}. Let \( E_{\bm{a}} \) be given by \eqref{1.3} and \( c'' \in \mathbb{Q} \setminus E_{\bm{a}} \). Then \( \mu_{M,D} \) is a spectral measure if and only if \( \rho_1^{-1} \in p\mathbb{Z} \).
\end{thm}
It is noteworthy that although \( \rho_1 \neq \rho_2 \), we can conclude that when \( c'' \in \mathbb{Q} \setminus E_{\bm{a}} \), the spectrality of \( \mu_{M, D} \) is independent of \( \rho_2 \). This is an intriguing phenomenon.
\begin{thm}\label{thm1.4}
	Suppose that \( \rho_1 \neq \rho_2 \). Let \(\mu_{M,D}\) be given by \eqref{1.1}, \eqref{1.2}, \eqref{x1.3} and \eqref{eq1.3}. Let \( E_{\bm{a}} \) be given by \eqref{1.3} and \( c'' = \frac{c_1}{c_2} \in E_{\bm{a}} \), where \( \gcd(c_1, c_2) = 1 \), \( c_2 \in p^{\ell} (\mathbb{Z} \setminus p\mathbb{Z}) \)  for some \( \ell \in \mathbb{N} \). Then \( \mu_{M,D} \) is a spectral measure if and only if \( \rho_1^{-1}, \rho_2^{-1} \in p\mathbb{Z} \) and \( p^{\ell+1} \mid (\rho_1^{-1} - \rho_2^{-1}) \).
\end{thm}
By synthesizing the results of Theorems \ref{thm1.5} to \ref{thm1.4}, we obtain the following theorem, which establishes the necessary and sufficient conditions for \( \mu_{M, D} \) to be a spectral measure.
\begin{thm}\label{thm1.6}
	Let \( \mu_{M,D} \) be given by \eqref{1.1}, \eqref{1.2}, \eqref{x1.3} and \eqref{eq1.3}, and let \( E_{\bm{a}} \) be given by \eqref{1.3}. Then \( \mu_{M,D} \) is a spectral measure if and only if \( M \) and \( D \) satisfy one of the following conditions:
	\begin{enumerate}
		\item[$(i)$] \( \rho_1 \neq \rho_2 \), \( \rho_1^{-1} \in p\mathbb{Z} \), and \( c'' \in \mathbb{Q}\setminus E_{\bm{a}} \).
		\item[$(ii)$] \( \rho_1\neq \rho_2 \), \( \rho_1^{-1}, \rho_2^{-1} \in p\mathbb{Z} \), and \( c'' = \frac{s}{t} \in E_{\bm{a}} \), \( t \in p^{\ell}(\mathbb{Z} \setminus p\mathbb{Z}) \), \( p^{\ell+1} \mid (\rho_1^{-1} - \rho_2^{-1}) \) for some \( \ell \in \mathbb{N} \).
		\item[$(iii)$] \( \rho_1 = \rho_2 \in p\mathbb{Z} \) and \( c \in \{ \frac{t}{s} : t \in p\mathbb{Z}, \gcd(s, t) = 1 \} \cup \{0\} \).
	\end{enumerate}
\end{thm}

The paper is organized as follows. In Section 2, we introduce key definitions and results that are essential for proving our main theorems and provide a foundation for the study of spectral measures. In Section 3, we establish the necessary conditions for \( \mu_{M, D} \) to be a spectral measure and prove Theorem \ref{thm1.2}. In Section 4, we analyze the case \( \rho_1 = \rho_2 \), considering both the spectrality and non-spectrality of \( \mu_{M, D} \), i.e., we give the proofs of Theorems \ref{thm1.7}, \ref{thm1.8} and \ref{thm1.5}. In Section 5, we examine the case \( \rho_1 \neq \rho_2 \), with a focus on proving Theorems \ref{thm1.3} and \ref{thm1.4}. Finally, in Section 6, we offer some examples.

\section{Preliminaries}

In this section, we present some preliminary lemmas and notations that will be used throughout the paper. Let \(\mu\) be a probability measure with compact support in \(\mathbb{R}^2\). The Fourier transform of \(\mu\) is defined in the usual manner:
\[
\widehat{\mu}(\xi) = \int_{\mathbb{R}^2} e^{-2\pi i \langle \xi, x \rangle} \, d\mu(x), \quad \xi \in \mathbb{R}^2.
\]
It is straightforward to show that \(\Lambda\) is an orthogonal set with respect to \(\mu\) if and only if
\[
\widehat{\mu}(\lambda - \lambda') = 0 \quad \text{for any} \quad \lambda \neq \lambda' \in \Lambda.
\]
In other words, the family \(E_{\Lambda} = \{ e^{-2\pi i \langle \lambda, x \rangle} : \lambda \in \Lambda \}\) is an orthogonal family in \(L^2(\mu)\) if and only if
$
(\Lambda - \Lambda) \setminus \{ \bm{0}\} \subset \mathcal{Z}(\widehat{\mu}),
$
where \(\mathcal{Z}(\widehat{\mu})\) denotes the zero set of the Fourier transform \(\widehat{\mu}\). We also say that \( \Lambda \) is a bi-zero set of \( \mu \).
Based on \eqref{1.1}, we can express \( \mu_{M,D} \) as an infinite convolution of the following form:
\begin{eqnarray}\label{2.1}
\mu_{M,D}&=& \delta_{M^{-1}D}*\delta_{M^{-2}D}*\delta_{M^{-3}D}*\cdots,
\end{eqnarray}
where $\delta_{E}=\frac{1}{\#E}\sum\limits_{e\in E}\delta_e$, $\delta_e$ is the Dirac measure at the point $e\in E$ and the convergence is in weak sense.

From \eqref{1.2} and \eqref{2.1}, it follows that
\begin{equation}\label{2.2}
	\mathcal{Z}(\widehat{\mu}_{M,D})=\bigcup_{k=1}^{\infty}(M^{*})^{k}\mathcal{Z}(\widehat{\delta}_{D})
=\bigcup_{k=1}^{\infty}(M^{*})^{k}\bigcup_{j=1}^{p-1}\left(\frac{j\bm{a}}{p}+\mathbb{Z}^{2}\right).
\end{equation}
When $\rho_1 \neq \rho_2$, recall that
\[
R = \begin{bmatrix} {d_{1,1}}^{-1} & c' \\ 0 & 1 \end{bmatrix} \quad \text{and} \quad D' = RD,
\]
where $c' = \frac{c}{d_{1,1}(\rho_1^{-1} - \rho_2^{-1})}$, and $D$ satisfies \eqref{eq1.3}. Then,
\begin{equation}\label{2.3}
	M' = RMR^{-1} = \begin{bmatrix} \rho_1^{-1} & 0 \\ 0 & \rho_2^{-1} \end{bmatrix} \quad \text{and} \quad
	\mathcal{Z}(\widehat{\delta}_{D'}) = (R^*)^{-1} \bigcup_{j=1}^{p-1} \left( \frac{j\bm{a} }{p}+\mathbb{Z}^2 \right).
\end{equation}
Since similarity transformations preserve the spectral properties of a self-affine measure \cite{DH1}, the spectral characteristics of \(\mu_{M,D}\) and \(\mu_{M',D'}\) are identical. By combining \eqref{2.2} and \eqref{2.3}, it is straightforward to verify that
\begin{equation}\label{2.4}
	\mathcal{Z}(\widehat{\mu}_{M',D'})= \bigcup_{k=1}^{\infty}\bigcup_{j=1}^{p-1}\left\{
	{\left({\begin{array}{*{20}{c}}
				\rho_{1}^{-k}d_{1,1}\left(\frac{ja_{1}}{p}+k_{1}\right)\\
				\rho_{2}^{-k}\left(\frac{ja_{2}-c'd_{1,1}ja_{1}}{p}+k_{2}-c'd_{1,1}k_{1}\right)\\
		\end{array}}\right)}:k_{1},k_{2}\in\mathbb{Z}
	\right\},
\end{equation}
where $\bm{a}=(a_{1},a_{2})^*\in \mathcal{E}_{p}.$ Define the function
\begin{equation*}
	Q_{\mu,\Lambda}(\xi)=\displaystyle{\sum_{\lambda\in \Lambda}}\left|\widehat{\mu}(\xi+\lambda)\right|^2.
\end{equation*}
The following theorem establishes a fundamental criterion for determining the spectrality of the measure \(\mu\) \cite{JP1}.
\begin{lemma}[{\cite{JP1}}]\label{lemma 2.1}
	Let $\mu$ be a Borel probability measure with compact support in $\mathbb{R}^{n}$, and let $ \Lambda \subset \mathbb{R}^{n} $ be a countable subset. Then
	\begin{itemize}
		\item[$(i)$] $E_{\Lambda}$ is an orthogonal family of $L^2(\mu)$ if and only if $Q_{\mu,\Lambda}(\xi)\leq1$ for $\xi\in \mathbb{R}^{n}$.
		\item[$(ii)$] $E_{\Lambda}$ is an orthogonal basis for $L^2(\mu)$ if and only if $Q_{\mu,\Lambda}(\xi)=1$ for $\xi\in \mathbb{R}^{n}$.
		\item[$(iii)$] $Q_{\mu,\Lambda}(\xi)$ has an entire analytic extension to $\mathbb{C}^{n}$ if $\Lambda$ is an orthogonal set of $\mu$.
	\end{itemize}
\end{lemma}

As an immediate consequence of Lemma \ref{lemma 2.1}, we obtain the following useful result.

\begin{lemma}[{\cite{DHL14}}] \label{lemma 2.2}
	Let \( \mu = \mu_0 * \mu_1 \) be the convolution of two probability measures \( \mu_i \), \( i = 0, 1 \), and they are not Dirac measures. Suppose that \( \Lambda \) is a bi-zero set of \( \mu_0 \), then \( \Lambda \) is also a bi-zero of \( \mu \), but cannot be a spectrum of \( \mu \).
\end{lemma}

 The following lemma is well-known.
\begin{lemma}{\rm{(\cite[Lemma 2.2]{AFL1}\label{lem2.8})}}
	Let \( I \) be a compact set in \( \mathbb{R}^{2} \). Then, the set \( F(I) \), which denotes the Fourier transforms of all Borel probability measures supported on the compact set \( I \subset \mathbb{R}^{2} \), is equi-continuous.
\end{lemma}

Next, we give the definition of compatible pair, which is essential for the construction of the spectrum.

\begin{defn}[{\cite{JP1}}]\label{de2.3}
	\rm Let $M\in M_{n}(\mathbb{Z})$  be an $n\times n$ expansive matrix with integer entries. Let $B,L \subset \mathbb{Z}^{n}$ be a finite set of integer vectors with $\#B = \#L=N$ and $\bm{0} \in B\cap L$. We say that $(M^{-1}B, L)$ forms a compatible pair (or $(M, B)$ is admissible, or $(M,B,L)$ forms a Hadamard triple)
	if the matrix
	\begin{align*}
		H=\frac{1}{\sqrt {N}}\left[e^{2\pi i  \langle M^{-1}b, \ell  \rangle}\right]_{b\in B, \ell\in L}
	\end{align*}
	is unitary, i.e., $HH^*=I$, where $H^*$ denotes the transposed conjugate of $H$.
\end{defn}
It is highly convenient to construct a family of orthogonal exponential functions for \( L^2(\mu) \) using a Hadamard triple. However, the primary challenge lies in verifying whether the constructed set \( \Lambda \) forms an orthonormal basis for \( L^2(\mu) \). For self-affine measures, we have the following well-known result.
\begin{lemma}[{\cite{DH1}}]\label{lem2.3}
	Let \((M, B, L)\) be a Hadamard triple. Then the self-affine measure \(\mu_{M, B}\) is spectral.
\end{lemma}

For $n\geq1$, define
$$\mu_n=\delta_{M^{-1}D}*\delta_{M^{-2}D}*\cdots*\delta_{M^{-n}D},$$
and
\begin{equation}\label{2.5}
\mu_{M,D}=\mu_n * \mu_{>n}.
\end{equation}
Let \(\Lambda_n\) be a bi-zero set of \(\mu_n\) with \(\bm{0} \in \Lambda_n\) and \(\Lambda_n \subset \Lambda_{n+1}\) for \(n \geq 1\). Then $\Lambda = \cup_{n=1}^{\infty} \Lambda_n$ is an orthogonal set of $\mu_{M,D}$.
The following proposition determines whether $\Lambda$ is a spectrum of $\mu_{M,D}$. Since the proof follows a similar approach as \cite[Theorem 2.3]{AHH}, we omit it.

\begin{lemma}[{\cite{AHH}}]\label{lemma 2.7}
 With the above notations, suppose that $\xi=(\xi_{1}, \xi_{2}), |\xi_{i}|\leq \frac{1}{3}$ and \( \{\alpha_n\}_{n=1}^{\infty} \) is an increasing sequence of integers.
\begin{itemize}
		\item[$(i)$] If \( \Lambda_{\alpha_n} \) is a spectrum of \( \mu_{\alpha_n} \) and
		\[
		\inf_{\lambda \in \Lambda_{\alpha_{n + l}} \setminus \Lambda_{\alpha_n} } |\mu_{>\alpha_{n + l}}(\lambda + \xi)|^2 \geq \beta > 0
		\]
		for \( n, l \geq 1 \), then \( \Lambda \) is a spectrum of \( \mu_{M,D} \).
		\item[$(ii)$]
		If
		\[
		\sup_{\lambda \in \Lambda_{\alpha_{n + 1}} \setminus \Lambda_{\alpha_n}} |\mu_{>\alpha_{n + 1}}(\lambda + \xi)|^2 \leq \beta_n
		\]
		for some sequence \( \{\beta_n\}_{n=1}^{\infty} \) with \( \sum_{n=1}^{\infty} \beta_n < \infty \), then \( \Lambda \) is not a spectrum of \( \mu_{M,D} \).
	\end{itemize}
\end{lemma}

In \cite{LDL}, Lu et al. established the following two lemmas, which are instrumental in our proof.
\begin{lemma}[{\cite[Lemma 2.5]{LDL}}]\label{lemma 2.8}
	Let \( \rho = \frac{s}{t} \in (0, 1) \) with \( s > 1 \) and \( \{ f_n(x) \}_{n=1}^{\infty} \) satisfy \( |f_n(x)| \leq 1 \) for any \( x \in \mathbb{R} \). Suppose that there exists a positive constant \( b_0 < 1 \) such that \( |f_n(x)| \leq b_0 \) for any \( |<x>| \geq \frac{1}{2t} \), where $<x>\in (-\frac{1}{2}, \frac{1}{2}]$ and $x-<x> \in \mathbb Z$. Then there is a constant \( c_0 > 0 \) such that
	\[
	\prod_{n=1}^{\infty} |f_n(\rho^n x)| \leq c_0 (\ln |x|)^{-\beta}, \quad |x| > 1,
	\]
	where \( \beta = \log_{a} b_0 > 0 \) with \( a = \log_{t} s\in (0, 1) \).
	
\end{lemma}

\begin{lemma}[{\cite[Lemma 2.6]{LDL}}]\label{lem2.7}
	Let \( \rho = \frac{s}{t} \in (0, 1) \) with \( s > 1 \) and \( \gcd(s, t) = 1 \). Suppose that a continuous non-negative real-valued function \( f(x_1, x_2, \cdots, x_n) \) with \( f(x_1, x_2, \cdots, x_n) \leq 1 \) satisfies the following conditions:
	\begin{enumerate}
		\item[$(i)$] There exist non-zero integers \( k_1, k_2, \cdots, k_n \) such that
		\[
		f(x_1, x_2, \cdots, x_n) = f(x_1 + k_1, x_2 + k_2, \cdots, x_n + k_n), \quad (x_1, x_2, \cdots, x_n) \in \mathbb{R}^n.
		\]
		\item[$(ii)$] There exists an integer \( \gamma \in \mathbb{N}^+ \) such that
		\[
		\left\{ (x_1, \cdots, x_n)^* : f(x_1, \cdots, x_n) = 1 \right\} \subset \left\{ \left( \frac{\ell}{\gamma}, x_2, \cdots, x_n \right)^* : \ell \in \mathbb{Z}, x_2, \cdots, x_n \in \mathbb{R} \right\}.
		\]
	\end{enumerate}
	Then, there exist two constants \( c_0, \beta > 0 \) such that
	\[
	\prod_{n=1}^{\infty} f\left( \rho^n x_1, x_{l,2}, \cdots, x_{l,n} \right) \leq c_0 (\ln |x_1|)^{-\beta}, \quad |x_1| > 1,
	\]
	for any sequence \( \{ x_{l,i} :2 \leq i \leq n, 1 \le l  \} \subset \mathbb{R} \).
	
\end{lemma}

At the end of this section, we present a conclusion that reflects the structure of the digit set satisfying condition \eqref{1.2}.
\begin{lemma}\label{lemma 2.9}
	Let \( D = \{ d_0, d_1, \cdots, d_{p-1} \} \subset \mathbb{Z}^2 \), and let \( E_{\bm{a}} \) be defined by \eqref{1.2} and \eqref{1.3}, respectively. Set $d_{i}=(d_{i,1}, d_{i,2})$, \( c'' = \frac{c_1}{c_2} \in \mathbb{Q} \setminus E_{\bm{a}} \) with \( \gcd(c_1, c_2) = 1 \) and $\mathcal{B}= \{ c_2 d_{0,1} + c_1 d_{0,2}, \cdots, c_2 d_{p-1,1} + c_1 d_{p-1,2} \}$, then
	\[
	\mathcal{B} = \{ 0, 1, \cdots, p-1 \} \pmod{p}\;\text{and}\;\gcd(\mathcal{B}) = 1.
	\]

\end{lemma}
\begin{proof}
Assume that $\mathcal{B} \neq \{ 0, 1, \cdots, p-1 \} \pmod{p}$. Since $\#B=p$, there exist $i \neq j \in \{0, 1, \cdots, p-1\}$ such that
$c_2 d_{i,1} + c_1 d_{i,2} \equiv c_2 d_{j,1} + c_1 d_{j,2} \pmod{p}.$
This implies that
$c_2 ( d_{i,1} - d_{j,1}) + c_1 ( d_{i,2} - d_{j,2} ) \equiv 0 \pmod{p}.$
By \( c'' \in\mathbb{Q}\backslash  E_{\bm{a}} \), we have $c_{2}a_{2}-c_{1}a_{1} \in p \mathbb Z$.
Since $p$ is prime, $\gcd(c_1, c_2) = 1$ and $1\leq a_{i}\leq p-1$, we have $\gcd(p, a_i) =\gcd(p, c_i) = 1$ and
	\[
	a_1 \left( d_{i,1} - d_{j,1} \right) +  a_2 \left( d_{i,2} - d_{j,2} \right) \equiv 0\pmod{p}.
	\]
By \(\mathcal{Z}\left(m_{D}\right) = \cup_{j=1}^{p-1} ( \frac{j \bm{a}}{p} + \mathbb{Z}^2 ) \), we obtain that $\sum_{k=0}^{p-1}e^{-2\pi i \langle \frac{m\bm{a}}{p}, d_{k}\rangle} =0$ for any $m=1,2,\cdots,p-1$.
Hence
\begin{align}\label{eq2.6}
  \sum_{m=1}^{p-1}\sum_{k=0}^{p-1}e^{-2\pi i \frac{m}{p}\langle \bm{a}, d_{k}-d_{j}\rangle}  =\sum_{m=1}^{p-1}\sum_{k=0}^{p-1}e^{-2\pi i \frac{m}{p}(a_1 \left( d_{k,1} - d_{j,1} \right) +  a_2 \left( d_{k,2} - d_{j,2} \right))} &=0, \notag \\
   i.e.,  \sum_{m=0}^{p-1}\sum_{k=0}^{p-1}e^{-2\pi i \frac{m}{p}(a_1 \left( d_{k,1} - d_{j,1} \right) +  a_2 \left( d_{k,2} - d_{j,2} \right))} =\sum_{k=0}^{p-1}\sum_{m=0}^{p-1}e^{-2\pi i \frac{m}{p}(a_1 \left( d_{k,1} - d_{j,1} \right) +  a_2 \left( d_{k,2} - d_{j,2} \right))} & =p,\notag \\
   2p+\sum_{k=0(k\neq i,j)}^{p-1}\sum_{m=0}^{p-1}e^{-2\pi i \frac{m}{p}(a_1 \left( d_{k,1} - d_{j,1} \right) +  a_2 \left( d_{k,2} - d_{j,2} \right))} & =p.
\end{align}
By $\gcd(p,a_{i})=1$ and $p$ is prime, it is easy to see
 \[
	\{m(a_1 \left( d_{k,1} - d_{j,1} \right) +  a_2 \left( d_{k,2} - d_{j,2} \right)), m=0,1,\cdots, p-1 \} = \{0\} \ or \ \{0,1,\cdots,p-1\} \pmod{p}.
	\]
Then for $k\neq i,j$, $\sum_{m=0}^{p-1}e^{-2\pi i \frac{m}{p}(a_1 \left( d_{k,1} - d_{j,1} \right) +  a_2 \left( d_{k,2} - d_{j,2} \right))}=0$ or $p$, while the right-hand side of \eqref{eq2.6} is equal to \(p\), resulting in a contradiction. Hence, 	$ \mathcal{B} = \{0, 1, \cdots, p-1\} \pmod{p}$.

By the fundamental theorem of arithmetic, let \( \gcd(\mathcal{B}) = h \), where \( h \) is a prime number. Now, we prove that $h=1$.
Clearly, \( h \neq p \). Then, we may assume that \( c_2 d_{i,1} + c_1 d_{i,2} = ht_i \). From \( \mathcal{B} = \{0, 1, \cdots, p-1\} \pmod{p} \), we obtain \( \{t_i\}_{i=0}^{p-1} = \{0, 1, \cdots, p-1\} \pmod{p} \). Furthermore,
	\[
	\sum_{k=0}^{p-1} e^{-2\pi i p^{-1} \langle h^{-1}(c_2, c_1)^*, d_k \rangle} =	\sum_{k=0}^{p-1} e^{-2\pi i p^{-1} t_{k}}= 0,
	\]
which implies that \( \frac{(c_2, c_1)^{*}}{ph} \in \mathcal{Z}(m_D) = \cup_{j=1}^{p-1} ( \frac{j \bm{a}}{p} + \mathbb{Z}^2 ) \). Therefore, \( (c_2, c_1)^* \in h\mathbb{Z}^2 \). Combining with \( \gcd(c_1, c_2) = 1 \), we have \( h = 1 \).
\end{proof}

\section{Necessary conditions}

In this section, we first establish that if $\mu_{M, D}$ is a spectral measure, then $\rho_{1}$ can be expressed as $\frac{s}{t}$ for some $s, t \in \mathbb{N}$ with $\gcd(s, t) = 1$ and $p \mid t$. Subsequently, under the assumptions that $\mu_{M, D}$ satisfies conditions \eqref{1.1}, \eqref{1.2}, \eqref{x1.3} and \eqref{eq1.3} with $\rho_{1} \neq \rho_{2}$, we further prove that if $\mu_{M, D}$ is a spectral measure, then $\rho_{1}^{-1} \in p\mathbb{Z}$ and $c'' \in \mathbb{Q}$ (Theorem \ref{thm1.2}). Prior to delving into the proofs, we introduce the concept of the selection map and present several lemmas that will be instrumental in the subsequent arguments of this section.

Denote $\Xi=\{0,1,\cdots,p-1\}$, $\Xi^{k}=\{\textbf{I} = i_{1}i_{2}\cdots i_{k}:\;i_{k}\in \Xi\}$ for $k \ge 1$ and $\Xi^{*}=\cup_{k=0}^{\infty}\Xi^{k} $, where $\Xi^{0}=\{\emptyset\}$. Notice that $\Xi^{*}$ can be thought of as a tree: the root is the empty set $\emptyset$, and for each node $\textbf{I}$, the nodes following it are $\textbf{I}\Xi $. In particular,
$\emptyset \textbf{I} = \textbf{I}$, $\textbf{I}0^{\infty} = \textbf{I}00 \cdots$, $ \; 0^k = 0 \cdots 0 \in \Xi^{k}.
$
\begin{defn}[\cite{DHL14}]\label{defn1.3} If \( p \mid t \), we call a map \( \gamma : \Xi^* \to \{ -1, 0, \cdots, t - 2 \} \) a selection mapping if
	\begin{enumerate}
		\item[$(i)$] \( \gamma(\emptyset) = \gamma(0^n) = 0 \) for all \( n \geq 1 \).
		\item[$(ii)$] For any \( \mathbf{I} = i_1 \cdots i_k \in \Xi^k \), \( \gamma(\mathbf{I}) \in i_k + p\mathbb{Z} \cap \mathcal{C} \), where \( \mathcal{C} = \{-1, 0, \cdots, t - 2\} \).
		\item[$(iii)$]  For any \( \mathbf{I} = i_1 \cdots i_k \in \Xi^* \), there exists \( \mathbf{J} \in \Xi^* \) such that \( \gamma \) vanishes eventually on \( \mathbf{IJ} 0^\infty \), i.e., \( \gamma(\mathbf{IJ}0^k) = 0 \) for sufficiently large \( k \).
	\end{enumerate}
\end{defn}
Let
\begin{equation*}
	\Xi^{\gamma} =\left\{ \textbf{I} = i_1 \cdots i_k \in \Xi^* : i_k \neq 0, \ \gamma(\textbf{I}0^n) = 0 \text{ for sufficiently large } n \right\} \cup \{\emptyset\},
\end{equation*}
where $\textbf{I}|_{k}$ denotes the prefix of $\textbf{I}$ with length $k$. Furthermore, we can define a mapping $\gamma^{*}$ from $\Xi^{\gamma} $ to $\mathbb{Z}$ by
\begin{equation*}
	\gamma^{*}(\textbf{I})=\sum_{k=1}^{\infty}t^{k-1}\gamma(\textbf{I}|_{k}),\quad\text{for any $\textbf{I}\in\Xi^{\gamma}$}.
\end{equation*}
\begin{lemma}[{\cite{D2}}]\label{lemma3.1}
	Let \( (\Lambda - \Lambda) \backslash \{ 0 \} \subset  \{ \rho^{-k}\ell : k \geq 1, \ell \in \cup_{i=1}^{p-1}(\frac{i}{p} +\mathbb{Z})\} \) for a real number \( 0 < \rho < 1 \). Then, \( \Lambda \) is an infinite set if and only if \( \rho \) is the \( r \)-th root of a fraction \( \frac{s}{t} \), where \( \gcd(s, t) = 1 \), $p\mid t$.
\end{lemma}
Dai $et\;al.$ \cite{DHL14} obtained the following result, which illustrates the relationship between maximal orthogonal sets and selection maps.
\begin{lemma}{\rm{(\cite[Theorem 4.5]{DHL14} \label{lemma3.3})}}
	Suppose that \( \rho = \frac{s}{t} \in (0, 1) \) with \( p \mid t \), and
	\[
	(\Lambda - \Lambda) \backslash \{ 0 \} \subset \bigcup_{j=1}^{\infty} \rho^{-j} \left( \bigcup_{i=1}^{p-1}\left(\frac{i}{p} +\mathbb{Z}\right) \right)
	\]
	with \( 0 \in \Lambda \). Then \( \Lambda \) is maximal if and only if there exists \( j_0 \geq 1 \) and a selection map \( \gamma \) such that
	$\Lambda = \rho^{-j_0} p^{-1} (\gamma^* (\Xi^\gamma)).$
\end{lemma}
For convenience, for any set \( \Lambda \subset \mathbb{R}^2 \), we define
\[
\Lambda^{(i)} = \left\{ \lambda_{k,i} : \lambda_{k} =\left(\lambda_{k,1}, \lambda_{k,2}\right)^* \in \Lambda \right\},
\]
where \( i = 1, 2 \). 
Similar to \cite[Proposition 3.3]{LDL}, we give the structure of the spectrum.
\begin{pro}\label{pro 3.4}
	Let \(\nu = \delta_1 * \delta_2 * \cdots * \delta_n * \cdots\) be a spectral measure satisfying \(\# \operatorname{spt} \nu = \infty\) and
	\[
	\emptyset \neq \mathcal{Z}(\widehat{\delta}_j) \subset \left\{
	\begin{pmatrix}
		\eta\rho^{-j}\left(\frac{i}{p} + k\right) \\
		x
	\end{pmatrix}
	: i\in \{1,2,\cdots,p-1\}, k \in \mathbb{Z}, x \in \mathbb{R}
	\right\} := S_j(\eta,\rho) \quad \text{for all } j \geq 1,
	\]
	where \(0 < \rho < 1\) and \(\eta \neq 0\). If \(\bm{0} \in \Lambda\) is a spectrum of \(\nu\), then \(\rho = \frac{s}{t}\) with \(t \in p\mathbb{Z}\), and there exist a selection map \(\gamma\) and a subset \(\Xi' \subset \Xi^\gamma\) such that
	\[
	\Lambda^{(1)} = \eta\rho^{-1} p^{-1} (\gamma^* (\Xi')) \subset \eta\rho^{-1} p^{-1}(\gamma^*(\Xi^\gamma)).
	\]
	Moreover, the spectrum \(\Lambda\) can be decomposed as
	$$
	\Lambda = \cup_{n=1}^\infty \Lambda_n=\cup_{n=1}^\infty \{\lambda(\mathbf{I}) : |\mathbf{I}| \leq n, \mathbf{I} \in \Xi'\},
	$$
	where \(\Lambda_n := \{\lambda(\mathbf{I}) : |\mathbf{I}| \leq n, \mathbf{I} \in \Xi'\}\) is a bi-zero set of \(\nu_n= \delta_1 * \delta_2 * \cdots * \delta_n\).
\end{pro}
\begin{proof}
By adapting the methodology used in the proof of Proposition 3.3 in \cite{LDL} and leveraging Lemmas \ref{lemma 2.2}, \ref{lemma3.1}, and \ref{lemma3.3}, the proof of this proposition follows straightforwardly.
\end{proof}
\begin{remark}	For any \(\lambda=(\lambda_{1},\lambda_{2})^* \in \Lambda\), by the above proposition, there exists a unique word \(\mathbf{I} \in \Xi'\) such that
	$
	\lambda_{1} = \eta \rho^{-1} p^{-1} \gamma^*(\mathbf{I}).
	$
	We can denote \(\lambda_{1} = \lambda_{1}(\mathbf{I})\) and \(\lambda = \lambda(\mathbf{I})\).
\end{remark}
It is straightforward to verify that the measure $\mu_{M,D}$, defined by conditions \eqref{1.1}, \eqref{1.2} and \eqref{x1.3}, satisfies the requirements of Proposition \ref{pro 3.4}. Thus, we conclude that if $\mu_{M,D}$ is a spectral measure, then $\rho_{1} = \frac{s}{t}$ with $t \in p\mathbb{Z}$. In the remainder of this section, we will assume that the digit set $D$ satisfies conditions \eqref{1.2}, \eqref{x1.3}and \eqref{eq1.3}. We first prove that if $\mu_{M, D}$ is a spectral measure, then $\rho_{1}\in p\mathbb{Z}$, followed by the proof of Theorem \ref{thm1.2}.

Let \( Q = \begin{bmatrix}
	d^{-1}_{1,1} & 0 \\
	0 & 1
\end{bmatrix} \). Then,
\begin{equation}\label{3.1}
	QMQ = \begin{bmatrix}
		\rho_1^{-1} & d^{-1}_{1,1} c \\
		0 & \rho_2^{-1}
	\end{bmatrix} := M'' \quad \text{and} \quad QD = \left\{
	\begin{pmatrix} 0 \\ 0 \end{pmatrix},
	\begin{pmatrix} 1 \\ 0 \end{pmatrix}, \cdots,
	\begin{pmatrix} \frac{d_{p-1,1}}{d_{1,1}} \\ d_{p-1,2} \end{pmatrix}
	\right\} := D''
\end{equation}
and $\mu_{M'', D''}$ shares the same spectral property as $\mu_{M, D}$. Write
\begin{equation}\label{3.3.1}
	\mu''_{>n}:=\delta_{{M''^{-(n+1)} D''}} *\delta_{{M''^{-(n+2)} D''}} * \cdots.
\end{equation}
It is readily verified that
\begin{equation*}
	\mathcal{Z}(\widehat{\delta}_{M''^{-j}D''})={{M''}^{*}}^j{Q^*}^{-1}\mathcal{Z}(m_{D}) \subset S_{j}(d_{1,1},\rho_{1}),
\end{equation*}
fulfilling the conditions of Proposition \ref{pro 3.4}. Consequently, we establish the following lemma.
\begin{pro}\label{pro 3.6}
	Let \( \rho_1 = \frac{s}{t} \) with \( s > 1 \) and \( t \in p\mathbb{Z} \). Suppose there exist \( \alpha, \beta > 0 \) such that for all \( n \geq 1 \) and \( \xi = (\xi_1, \xi_2)^* \in \mathbb{R}^2 \) with \( |\rho_1^n \xi_1| > 1 \),
	\[
	|\widehat{\mu''}_{>n}(\xi)| \leq \alpha \left( \ln \left| \rho_1^n \xi_1 \right| \right)^{-\beta}.
	\]
Then \( \mu_{M,D} \) is not a spectral measure.
\end{pro}
\begin{proof}
Suppose that \( \mu_{M,D} \) is a spectral measure and that \( \bm{0} \in \Lambda \) is a spectrum of \( \mu_{M'',D''} \).  From Proposition \ref{pro 3.4}, we know that there exists a selection map \( \gamma \) and a subset \( \Xi^{'} \subset \Xi^{\gamma} \) such that \( \Lambda^{(1)} = d_{1,1}\rho_{1}^{-1} p^{-1} (\gamma^* (\Xi^{'})) \subset d_{1,1}\rho_{1}^{-1} p^{-1} (\gamma^*(\Xi^{\gamma})) \). Additionally, \( \Lambda = \cup_{n=1}^{\infty} \Lambda_n \), where \( \Lambda_n := \{ \lambda(\textbf{I}) : |\textbf{I}| \leq n, \textbf{I} \in \Xi^{'} \} \).
	For any \( n \geq 1 \) and \( \lambda=(\lambda_{1},\lambda_{2})^* \in \Lambda \setminus \Lambda_n \), there exists \( \textbf{I} \in \Xi^{'} \) with \( |\textbf{I}| \geq n + 1 \) such that
	$ \lambda_{1} = \lambda_{1}(\textbf{I}) = d_{1,1}\rho_{1}^{-1} p^{-1} \gamma^{*}(\textbf{I}).$
	This implies that
	\begin{equation}\label{3.3}
		\left|\lambda_{1}(\textbf{I})\right| \geq \frac{|d_{1,1}|t}{ps} \left( t^n - (t - 2)t^{n-1} - \cdots - (t - 2) \right) \geq \frac{|d_{1,1}|(t^{n}+1)}{ps}.
	\end{equation}
	Let \( K > \beta^{-1} \) and \( \alpha_n = n^{K} \). Then, there exists a positive integer \( n_0 \) such that
	$ (n + 1)^{K} \leq \left( 1 + \frac{1}{2} \log_t^{s}  \right)n^{K} \;\text{for any } n \geq n_0.$
	For any \( \xi = (\xi_1, \xi_2)^* \in [ 0,| \frac{d_{1,1}}{p s} |]^2 \) and \( n^{K} < |\textbf{I}| \leq (n+1)^{K} \), it follows from \eqref{3.3} that
	\begin{align*}
		|\widehat{\mu''}_{>\alpha_{n+1}}(\xi+\lambda(\textbf{I}))| &\leq \alpha \left( \ln \rho_1^{(n+1)^{K}}\left|  \xi_1 +\lambda_{1}(\textbf{I})\right| \right)^{-\beta}\\
		&\leq \alpha \left( \ln \rho_1^{(n+1)^{K}}|d_{1,1}|\frac{t^{n^{K}}+1}{ps} \right)^{-\beta}\\
		&\le \alpha \left( [(n+1)^{K}\log_{t} \rho_{1} +n^{K}+\log_{t}|d_{1,1}| -\log_{t}ps]\ln t \right)^{-\beta}\\
		&\le \frac{C}{n}
	\end{align*}
	for some constants \( C, N_0 > 0 \), and for all \( n \geq N_0 \geq n_0 \). Since $ \sum_{n=1}^{\infty} \frac{C^2}{n^2} < \infty
	$ and by applying Lemma \ref{lemma 2.7} (ii), we conclude that \( \Lambda \) cannot be a spectrum of \( \mu_{M'',D''} \), which leads to a contradiction. This completes the proof that
	$\mu_{M, D}$ is not a spectral measure.
\end{proof}
 \begin{lemma}\label{lemma 3.7}
	Let \( \mu_{M,D} \) be given by \eqref{1.1}, \eqref{1.2}, \eqref{x1.3}and \eqref{eq1.3}. If $\mu_{M,D}$ is a spectral measure, then $\rho_{1}^{-1}\in p\mathbb{Z}$.
\end{lemma}
\begin{proof}
If $\mu_{M,D}$ is a spectral measure, $\mu_{M'',D''}$ is also a spectral measure. From Proposition \ref{pro 3.4}, it follows that \(\rho_{1}^{-1} = \frac{t_{1}}{s_{1}}\), where \(t_{1} \in p\mathbb{Z}\) and \(\gcd(s_{1}, t_{1}) = 1\). Therefore, we only need to prove that \( s_{1} = 1 \). By contradiction, assume that \( s_{1} > 1 \). Let \(\xi = (\xi_1, \xi_2)^*\), and define $f_{n}(\xi_{1}):=|m_{{M''}^{-n}D''}((\rho_{1}^{-n}\xi_{1},\xi_2)^{*})|$. Then, it follows that $\widehat{\mu}_{M'',D''}(\xi)=\Pi_{j=1}^{\infty}f_{j}(\rho_{1}^{j}\xi_{1}).$ Note that
	\[
	{M}''^{-1} = \begin{bmatrix} \rho_1 & -{d_{1,1}}^{-1}c \rho_1 \rho_2 \\ 0 & \rho_2 \end{bmatrix} \quad \text{and} \quad M''^{-2} = \begin{bmatrix} \rho_1^2 & -{d_{1,1}}^{-1}c \rho_1 \rho_2 (\rho_1 + \rho_2) \\ 0 & \rho_2^2 \end{bmatrix}.
	\]
	For convenience, we can denote
	\[
	M''^{-n} = \begin{bmatrix} \rho_1^n & \vartheta_n \\ 0 & \rho_2^n \end{bmatrix},
	\]
	where \( \vartheta_n \) is determined by \( \rho_{1} \), \( \rho_{2} \), and \( c{d_{1,1}}^{-1}
	\).
	Furthermore, we can estimate \( f_{n}(\xi_{1}) \) as
	\begin{align*}
		f_{n}(\xi_{1})&=\frac{1}{p}\left|\sum_{j=0}^{p-1}e^{-2\pi i\left\langle \left(\rho_{1}^{n}\frac{d_{j,1}}{d_{1,1}}+\vartheta_nd_{j,2},\rho_{2}^{n}d_{j,2}\right)^*, \left(\rho_{1}^{-n}\xi_{1},\xi_2\right)^* \right\rangle}\right|\\
		&=\frac{1}{p}\left|\sum_{j=0}^{p-1}e^{-2\pi i\left(\xi_{1}\frac{d_{j,1}}{d_{1,1}}+\vartheta_nd_{j,2}\rho_{1}^{-n}\xi_{1}+\rho_2^{n}\xi_{2}d_{j,2}\right)}\right|\\
		&\le\frac{1}{p}\left|1+e^{-2\pi i\xi_{1}}\right|+\frac{p-2}{p}.
	\end{align*}
	It is easy to see that if \( |\langle\xi_1\rangle| \geq \frac{1}{2t_1} \), then $ f_{n}(\xi_{1})<1.$  For each $n>1$, by Lemma \ref{lemma 2.8}, there exist \( \alpha, \beta > 0 \) such that
	\[
	| \widehat{\mu''}_{>n}(\xi) | = \prod_{j=1}^{\infty} f_{n+j}(\rho_1^{n+j} \xi_1) \leq \alpha \left( \ln |\rho_1^n \xi_1| \right)^{-\beta}
	\]
	for any \( |\rho_1^n \xi_1| > 1 \). It follows from Proposition \ref{pro 3.6} that \( \mu_{M'', D''} \) is not a spectral measure, a contradiction. Hence \( \rho_1^{-1} \in p\mathbb{Z} \).
\end{proof}
At the end of this section, we will prove that when \(\rho_1 \neq\rho_2 \) and \( c''\notin\mathbb{Q} \), \( \mu_{M,D} \) is not a spectral measure. From \eqref{2.2}, we define \( \mathcal{Z}_k := {M'}^{*^k} \mathcal{Z}(m_{D'}) \), which will be used frequently later.
\begin{lemma}\label{lemma 3.8}
	Let \(\rho_{1}^{-1} = t_1 \in p\mathbb{Z}\), and let \(\mu_{M', D'}\) be a spectral measure with spectrum \(\Lambda\). Then, there exists \(\lambda \in \Lambda\) such that \((\Lambda - \lambda) \cap \mathcal{Z}_k \neq \emptyset\) for all \(k = 1, 2, 3\), where \(\mathcal{Z}_k = {M'}^{*^k} \mathcal{Z}(m_{D'})\).
\end{lemma}
\begin{proof}
	From Proposition \ref{pro 3.4}, we obtain the decomposition \(\Lambda = \cup_{n=1}^\infty \Lambda_n\), where each \(\Lambda_n := \{\lambda(\mathbf{I}) : |\mathbf{I}| \leq n, \mathbf{I} \in \Xi'\}\) forms a bi-zero set for the measure \(\mu'_n\). If \((\Lambda - \Lambda) \cap \mathcal{Z}_3 = \emptyset\), then \((\Lambda - \Lambda) \setminus \{\bm{0}\} \subseteq \cup_{k \neq 3} \mathcal{Z}_k\). By Lemma \ref{lemma 2.2}, it follows that \(\Lambda\) cannot be a spectrum of \(\mu_{M', D'}\). This contradiction enables us to choose an element \(\lambda \in \Lambda\) such that \((\Lambda - \lambda) \cap \mathcal{Z}_3 \neq \emptyset\). We claim that \(\lambda(\mathbf{I}) - \lambda(\mathbf{J}) \in \mathcal{Z}_s \; \text{for any } \mathbf{I}, \mathbf{J} \in \Xi'\),
	where \( s = \min \left\{ i : \mathbf{I}|_i \neq \mathbf{J}|_i \right\}\). In fact
	\begin{align*}
		\lambda_{1}(\mathbf{I}) - \lambda_{1}(\mathbf{J}) &=  d_{1,1}p^{-1}t_1 (\gamma^*(\mathbf{I}) - \gamma^*(\mathbf{J})) \\
		&= d_{1,1}p^{-1}t_1 \left( \sum_{n=1}^{\infty} \left( \gamma(\mathbf{I}0^\infty | n) - \gamma(\mathbf{J}0^\infty | n) \right) t_1^{n-1} \right) \in d_{1,1}t_1^s \bigcup_{i=1}^{p-1} \left( \frac{i}{p} + \mathbb{Z} \right).
	\end{align*}
	This shows that the claim is true.
	
	If $\{ i \mathbf{I} \in \Xi' : i = 1, 2, \cdots, p-1 \} = \emptyset$, then it follows that $\lambda(\mathbf{I}) - \lambda(\mathbf{J}) \in \cup_{k=2}^{\infty} \mathcal{Z}_k$ for any $\mathbf{I}, \mathbf{J} \in \Xi'$. Using Lemma \ref{lemma 2.2} again, we conclude that $\Lambda$ is not a spectrum of $\mu_{M', D'}$, which leads to a contradiction. Thus, $\{ i \mathbf{I} \in \Xi' : i = 1, 2, \cdots, p-1 \} \neq \emptyset$. Since \( \bm{0} \in \Lambda \) and by the claim, we have \( \lambda(\mathbf{I}) \in \mathcal{Z}_1 \) for any \( \mathbf{I} \in \Xi' \) with \( \mathbf{I}|_1 \neq 0 \). Therefore,
	\begin{align*}
		m_{D'} \left( {M'^*}^{ - 1} \lambda(\mathbf{I}) \right)=\frac{1}{p} \sum_{j=0}^{p-1} e^{-2\pi i\left( t_{1}^{-1}\lambda_{1}(\textbf{I})\left({d_{1,1}}^{-1}d_{j,1}+c'd_{j,2}\right)+\rho_2\lambda_{2}(\textbf{I})d_{j,2} \right)}= 0.
	\end{align*}
	Let \( \xi_0 := d_{1,1}p^{-1}t_1\left(  t_1, -\rho_2^{-1} c' \right)^* \), then
	\[
	m_{D'} \left( {M'^{*}}^{-1} (\xi_0 + \lambda(\mathbf{I})) \right) = m_{D'} \left( {M'^*}^{-1} \lambda(\mathbf{I}) \right) = 0.
	\]
	Furthermore, we have the following equation:
	\begin{equation}\label{3.4}
		\sum_{\textbf{I} \in \Xi', \textbf{I}|_{1}\neq 0} \left| \widehat{\mu}_{M', D'}(\xi_0 + \lambda(\textbf{I})) \right|^2
		=\sum_{\textbf{I} \in \Xi', \textbf{I}|_{1}\neq 0} \left| m_{D'}\left({M'^*}^{-1}(\xi_0 + \lambda(\textbf{I}))\right) \right|^2 \left| \widehat{\mu'}_{> 1}(\xi_0 + \lambda(\textbf{I})) \right|^2=0,
	\end{equation}
	where \( \mu'_{> 1} \) is defined similarly to \eqref{3.3.1}.
	We assert that \( \Lambda_\textbf{I} := \{ \lambda(\textbf{IJ}) : \textbf{IJ} \in \Xi' \} \) is a bi-zero set of \( \mu'_{> n} \) for any \( \textbf{I} \in \Xi^{n} \). For any \( \textbf{IJ} = \textbf{IJ}' \in \Lambda_\textbf{I} \), by the first Claim, it follows that
	\[
	\lambda(\textbf{IJ}') - \lambda(\textbf{IJ}) \in \mathcal{Z}_s \subset \mathcal{Z}(\widehat{\mu'}_{> n})
	\]
	where \( s = n + \min \{ j : \textbf{IJ}|_j \neq \textbf{IJ}'|_j \} \). Therefore, \( \Lambda_\textbf{I} \) is a bi-zero set of \( \mu'_{> n} \), and
	\begin{equation}\label{3.5}
		\sum_{\lambda \in \Lambda_\textbf{I}} \left| \widehat{\mu'}_{> n}(\xi + \lambda) \right|^2 \leq 1
	\end{equation}
	for any \( \xi \in \mathbb{R}^2 \). Assume that \( \{ 0 i \textbf{I} \in \Xi' : i = 1, 2,\cdots,p-1 \} = \emptyset \). Then, from \eqref{3.4}, it follows that
	\[
	\sum_{\lambda \in \Lambda} \left| \widehat{\mu}_{M', D'}(\xi_0 + \lambda) \right|^2 = \sum_{\textbf{I} \in \Xi'} \left| \widehat{\mu}_{M', D'}(\xi_0 + \lambda(\textbf{I})) \right|^2 = \sum_{\textbf{I} = 00 \textbf{J} \in \Xi'} \left| \widehat{\mu}_{M', D'}(\xi_0 + \lambda(\textbf{I})) \right|^2.
	\]
	From \eqref{3.5}, the expression above is bounded as follows
	\begin{align*}
		&\sum_{\textbf{I} = 00 \textbf{J} \in \Xi'} \left| m_{D'}\left({M'^*}^{- 2} \left(\xi_0 + \lambda(\textbf{I})\right)\right) \right|^2 \left| \widehat{\mu'}_{> 2} (\xi_0 + \lambda(\textbf{I})) \right|^2\\
		&\leq \max_{{\textbf{I} = 00 \textbf{J} \in \Xi'}} \left| m_{D'}\left({M'^*}^{- 2}\left(\xi_0 + \lambda(\textbf{I})\right)\right) \right|^2 \\
		&\leq \max_{{\textbf{I} = 00 \textbf{J} \in \Xi'}} \frac{1}{p^{2}} \left( \left|1+ e^{-2\pi i \left(t_{1}^{-2}\lambda_{1}(\textbf{I})+p^{-1}d_{1,1}\right)} \right| + p-2 \right)^2 < 1.
	\end{align*}
	According to Lemma \ref{lemma 2.9}, we obtain \( p\nmid d_{1,1} \), and thus the last inequality holds. From Lemma \ref{lemma 2.1}, \( \Lambda \) is not a spectrum of \( \mu_{M', D'} \), which leads to a contradiction.
\end{proof}
In the following, we will prove Theorem \ref{thm1.2}.
\begin{proof}[\textbf{Proof of Theorem \ref{thm1.2}}]
	Suppose that \( \mu_{M', D'} \) is a spectral measure. By Lemma \ref{lemma 3.7}, \( \rho^{-1}_1 = t_1 \in p\mathbb{Z} \). According to Lemma \ref{lemma 3.8}, we can assume that \( \bm{0} \in \Lambda \) is a spectrum with \( \Lambda \cap \mathcal{Z}_j \neq \emptyset \) for \( j = 1, 2, 3\). For each \( n = 1, 2, 3 \) and \( \lambda_n \in \mathcal{Z}_n \cap \Lambda \), from \eqref{2.4}, let
	\[
	\lambda_n = \frac{1}{p}
	{\left({\begin{array}{*{20}{c}}
				t_1^{n} d_{1,1}\left( s_{n,1} + pk_{n,1} \right)\\
				\rho^{-n}_2 \left( s_{n,2} + pk_{n,2} - c'd_{1,1} \left( s_{n,1} + pk_{n,1} \right) \right)\\
		\end{array}}\right)},
	\]
where $s_{n}=(s_{n,1},s_{n,2})^*\in \mathcal{E}_{p}$ and $k_{n}=(k_{n,1},k_{n,2})^*\in \mathbb Z^{2}$. For any \( n > \ell \), the difference \( \lambda_{n} - \lambda_{\ell} \) can be written as
	\[ \frac{1}{p}
	{\left({\begin{array}{*{20}{c}}
				d_{1,1}t_1^{\ell} \left(t_1^{n-\ell}\left( s_{n,1} + pk_{n,1} \right)-\left( s_{\ell,1} + pk_{\ell,1} \right)\right)\\
				\rho^{-\ell}_2 \left(\rho^{\ell-n}_2\left( s_{n,2} + pk_{n,2} - c'd_{1,1} \left( s_{n,1} + pk_{n,1} \right) \right)-\left( s_{\ell,2} + pk_{\ell,2} - c'd_{1,1} \left( s_{\ell,1} + pk_{\ell,1} \right) \right)\right)\\
		\end{array}}\right)}.\]
	Since \( \lambda_{n} - \lambda_{\ell} \in \cup_{j=1}^{\infty} \mathcal{Z}_j \), it follows from the first component that $ \lambda_{n} - \lambda_{\ell} \in \mathcal{Z}_{\ell} = {M'^*}^\ell {R^*}^{-1} \mathcal{Z}(m_D)$. Therefore, the second component satisfies the following relationship
	\begin{align*}
		&\rho^{\ell-n}_2\left( s_{n,2} + pk_{n,2} - c' d_{1,1}\left( s_{n,1} + pk_{n,1} \right) \right)-\left( s_{\ell,2} + pk_{\ell,2} - c'd_{1,1} \left( s_{\ell,1} + pk_{\ell,1} \right) \right)\\
		=&s_{n\ell,2} + pk_{n\ell,2} -c'd_{1,1}\left( t_1^{n-\ell}\left( s_{n,1} + pk_{n,1} \right)-\left( s_{\ell,1} + pk_{\ell,1} \right)\right).
	\end{align*}
	We can then derive the following equation
	\begin{align}\label{3.7}
		\nonumber
		\left(\frac{\rho_2^{-1}}{t_{1}}\right)^{n-\ell}&=\frac{\frac{s_{n\ell,2}  + pk_{n\ell,2} + s_{\ell,2} + pk_{\ell,2}}{t_{1}^{n-\ell}}-c'd_{1,1}\left( s_{n,1} + pk_{n,1} \right)}{s_{n,2} + pk_{n,2} - c'd_{1,1} \left( s_{n,1} + pk_{n,1} \right)}\\
		&=1+\frac{\frac{s_{n\ell,2} + pk_{n\ell,2} + s_{\ell,2} + pk_{\ell,2}}{t_{1}^{n-\ell}}- s_{n,2} - pk_{n,2} }{s_{n,2} + pk_{n,2} - c'd_{1,1} \left( s_{n,1} + pk_{n,1} \right)}.
	\end{align}
	Note that \( \frac{s_{n\ell,2} + p k_{n\ell,2} + s_{\ell,2} +p k_{\ell,2}}{t_1^{n-\ell}} \in \mathbb{Q} \). This implies that for any \( \ell_1<\ell_2 < n \),
	\[
	\frac{\left( \frac{\rho_2^{-1}}{t_1} \right)^{n - \ell_1} - 1}{\left( \frac{\rho_2^{-1}}{t_1} \right)^{n - \ell_2} - 1} \in \mathbb{Q}.
	\]
	Set \( \ell_1 = 1 \), \( \ell_2 = 2 \), and \( n = 3 \). We obtain
	\[
	\frac{\rho_2^{-1}}{t_1} + 1 \in \mathbb{Q}.
	\]
	Returning to \eqref{3.7}, we obtain \( c', c'' \in \mathbb{Q} \), which leads to a contradiction of the initial assumption. Hence, \( \mu_{M', D'} \) is not a spectral measure.
\end{proof}

\section{The case $\rho_1 = \rho_2 $}

In this section, we explore the scenario where $\rho_1 = \rho_2:=\rho$ and present the proofs of Theorems \ref{thm1.7}, \ref{thm1.8}, and \ref{thm1.5}. Initially, we establish a necessary and sufficient condition for the existence of an infinite orthogonal set of exponential functions in $L^2(\mu_{M, D})$. When such an infinite orthogonal set does not exist in $L^2(\mu_{M, D})$, we not only provide an estimate for the number of orthogonal exponential functions but also determine its maximal cardinality. Finally, we obtain a necessary and sufficient condition for $\mu_{M, D}$ to be a spectral measure.
\begin{lemma}\label{lemma4.3}
	Suppose that \( \sigma \in \{  ( \frac{s}{t} )^{\frac{1}{r}} : s, t, r \in \mathbb{N}, \gcd(s,t)=1\} \) satisfies the equation
	\[
	k_1 \sigma^i + k_2 \sigma^j = k_3 \sigma^l,
	\]
	where \( i, j, l \) are nonnegative integers. Then the following statements hold.
	\begin{itemize}
		\item[$(i)$] If \( k_1, k_2, k_3 \in \mathbb{Z} \setminus \{0\} \), then \( i \equiv j \equiv l \pmod{r} \).
		\item[$(ii)$]If \( k_1, k_2, k_3 \in \mathbb{Z} \setminus p\mathbb{Z} \), and \( s \in p\mathbb{Z} \) or \( t \in p\mathbb{Z} \), where \(p \) is a prime, then at least two of the integers \( i, j, l \) must be equal.
	\end{itemize}
\end{lemma}
\begin{proof}
	$(i)$ This result is an immediate consequence of \cite[Lemma 2.5]{D2}. $(ii)$ Suppose, for the sake of contradiction, that \( i, j, l \) are distinct nonnegative integers. Without loss of generality, we assume that \( i > j > l \). By $(i)$, we can express \( i \), \( j \), and \( l \) as follows: $
	i = i_1r + k, \; j = j_1r + k, \; l = l_1r + k,
	$
	where \( i_1 > j_1 > l_1 \) and \( 0 \leq k \leq r - 1 \). From \( k_1 \sigma^i + k_2 \sigma^j = k_3 \sigma^l \), we obtain the equation
	\[
	k_{1} s^{i_{1}-l_{1}} + k_{2} s^{j_{1}-l_{1}} t^{i_{1}-j_{1}} = k_{3} t^{i_{1}-l_{1}},
	\]
	which can be easily verified to be impossible, thus leading to a contradiction and completing the proof.
\end{proof}
\begin{proof}[\textbf{Proof of Theorem \ref{thm1.7}}]
$"\Rightarrow".$ Suppose that \( L^2(\mu_{M,D}) \) admits an infinite orthogonal set \( E_{\Lambda} \). Then
	\[
	\Lambda \setminus \{\bm{0}\}, (\Lambda - \Lambda) \setminus \{\bm{0}\} \subset\mathcal{Z}(\widehat{\mu}_{M,D})=\bigcup_{j=1}^{p-1} \mathcal{H}_{j} ,
	\]
	where
\begin{equation*}
		\mathcal{H}_j=\bigcup_{k=1}^{\infty}(M^*)^{k}\left(\frac{j\bm{a}+p\mathbb{Z}^{2}}{p}\right).
\end{equation*}
	We claim that \( \Lambda^{(1)} \) is infinite. Suppose on the contrary that \( \Lambda^{(1)} \) is finite, then \( \Lambda^{(2)} \) must be infinite. By the pigeonhole principle, there exist \( \lambda_1 \neq\lambda_2 \in \Lambda \setminus \{\bm{0}\} \) such that \( \lambda_{1,1} = \lambda_{2,1} \), where \( \lambda_i = (\lambda_{i,1}, \lambda_{i,2})^* \) for \( i = 1, 2 \). Thus,
	\[
	\lambda_1 - \lambda_2 = {\left({\begin{array}{*{20}{c}}
				0\\
				\lambda_{1,2} - \lambda_{2,2}
		\end{array}}\right)} \in \bigcup_{j=1}^{p-1} \mathcal{H}_{j} ,
	\]
	which leads to a contradiction.
	Thus, \( \Lambda^{(1)} \) is an infinite set satisfying
	\[
	\Lambda^{(1)} \setminus \{0\}, \left(\Lambda^{(1)} - \Lambda^{(1)}\right) \setminus \{0\} \subset \bigcup_{j=1}^{\infty} \rho^{-j} \left(\frac{\{1,2,\cdots,p-1\}}{p}+\mathbb{Z}\right).
	\]
	According to Lemma \ref{lemma3.1}, we have \( \rho^{-1} = (\frac{t}{s})^{\frac{1}{r}} \) for some \( s, t, r \in \mathbb{N} \) with \( t \in p\mathbb{Z} \) and \( \gcd(s, t) = 1 \).
	
	In the following, we will prove that \( c = \kappa \rho^{-1} \), where \( \kappa \in \mathbb{Q} \).
	Since \( \# \Lambda = \infty \), without loss of generality, there exist \( \lambda_1 \neq  \lambda_2 \in \Lambda \setminus \{\bm{0}\} \) such that \( \lambda_1, \lambda_2 \in \mathcal{H}_1 \). From the definition of \( \mathcal{H}_1 \), we can write
	\[
	\lambda_1 = \rho^{-j_{1}} \begin{pmatrix} \frac{a_{1}}{p}+k_{j_{1},1} \\ \frac{j_{1}c\rho a_{1}+a_{2}}{p}+j_{1}c\rho k_{j_{1},1}+k_{j_{1},2} \end{pmatrix}, \quad
	\lambda_2 = \rho^{-j_{2}} \begin{pmatrix} \frac{a_{1}}{p}+k_{j_{2},1} \\ \frac{j_{2}c\rho a_{1}+a_{2}}{p}+j_{2}c\rho k_{j_{2},1}+k_{j_{2},2} \end{pmatrix},
	\]
	and
	\begin{equation}\label{4.1}
		\lambda_1 - \lambda_2 = \rho^{-j_{3}} \begin{pmatrix} \frac{ia_{1}}{p}+k_{j_{3},1} \\ \frac{j_{3}c\rho ia_{1}+ia_{2}}{p}+j_{3}c\rho k_{j_{3},1}+k_{j_{3},2} \end{pmatrix},
	\end{equation}
	where $i\in\{1,2,\cdots,p-1\}$, \( k_{j_s} = (k_{j_s,1}, k_{j_s,2})^* \in \mathbb{Z}^2 \) for \( s = 1, 2, 3 \), and \( \bm{a} = (a_1, a_2)^*\in  \mathcal{E}_{p} \).
	By \eqref{4.1}, we obtain
	\begin{equation}\label{4.2}
		\rho^{-j_{1}}\left(a_{1}+pk_{j_{1},1}\right)-\rho^{-j_{2}}\left(a_{1}+pk_{j_{2},1}\right) = \rho^{-j_{3}}\left(ia_{1}+pk_{j_{3},1}\right),
	\end{equation}
	and
	\begin{align}\label{4.3}
		\nonumber&\rho^{-j_{1}}\left(j_{1}c\rho a_{1}+a_{2}+p\left(j_{1}c\rho k_{j_{1},1}+k_{j_{1},2}\right)\right)-\rho^{-j_{2}}\left(j_{2}c\rho a_{1}+a_{2}+p\left(j_{2}c\rho k_{j_{2},1}+k_{j_{2},2}\right)\right)\\
		&=\rho^{-j_{3}}\left(j_{3}c\rho ia_{1}+ia_{2}+p\left(j_{3}c\rho k_{j_{3},1}+k_{j_{3},2}\right)\right).
	\end{align}
	Note that \( \rho^{-1} = ( \frac{t}{s} )^{\frac{1}{r}} \) and \( t \in p\mathbb{Z} \). By Lemma \ref{lemma4.3}, we conclude that \( j_{1} \equiv j_{2} \equiv j_{3} \pmod{r} \), and at least two of \( j_{1}, j_{2}, j_{3} \) are equal. Since \( p(k_{j_{1},1} - k_{j_{2},1}) \in p\mathbb{Z} \) and \( ia_{1} + pk_{j_{3},1} \notin p\mathbb{Z} \), it follows from \eqref{4.2} that \( j_{1}, j_{2}, j_{3} \) are not all equal. Without loss of generality, we assume that \( j_{1}\neq j_{2}=j_{3} \). Thus, combining with \eqref{4.2} and \eqref{4.3}, we obtain that
	\[
	c=\rho^{-1}\frac{\left(a_{2}+pk_{j_{1},2}\right)-\rho^{j_{1}-j_{3}}\left(ia_{2}+pk_{j_{3},2}+a_{2}+pk_{j_{2},2}\right)}{(j_{3}-j_{1})\left(a_{1}+pk_{j_{1},1}\right)}\in\rho^{-1}\mathbb{Q}.
	\]
	The necessity follows.
	
$"\Leftarrow".$  Suppose that \( \rho^{-1} =( \frac{t}{s})^{\frac{1}{r}} \) and \( c = \kappa \rho^{-1} \), where \( s, t, r \in \mathbb{N} \), \( t \in p\mathbb{Z} \), \( \gcd(s, t) = 1 \), and \( \kappa \in \mathbb{Q} \). Then \( \gcd(s, p) = 1 \).
	Let \( \kappa = \frac{v}{u} \), where \( u \in \mathbb{N} \), \( v \in \mathbb{Z} \), and \( \gcd(u, v) = 1 \). Define
	\[
	\Lambda = \left\{ \frac{t^{pu\ell} \bm{a}}{p} : \ell \in \mathbb{N} \right\}, \quad \text{where} \; \bm{a} = (a_{1}, a_{2})^{*}\in \mathcal{E}_{p} .
	\]
	Note that \( t = s\rho^{-r} \), we have
	$\frac{t^{pu\ell} \bm{a}}{p}=\rho^{-rpu\ell}\frac{s^{pu\ell} \bm{a}}{p}$.
	Furthermore, since
\begin{equation*}
	M^{*r} = {\begin{bmatrix}
		\rho^{-1} & 0 \\
		c & \rho^{-1}
	\end{bmatrix}}^{r}
=\begin{bmatrix}
		\rho^{-r} & 0 \\
		rc\rho^{-r+1} & \rho^{-r}
	\end{bmatrix}
=\begin{bmatrix}
		\rho^{-r} & 0 \\
		r\kappa \rho^{-r} & \rho^{-r}
	\end{bmatrix} ,
\end{equation*}
we get
	\begin{equation}\label{4.4}
		\frac{t^{pu\ell} \bm{a}}{p}= {M^{*}}^{rpu\ell} \begin{pmatrix} \frac{a_{1}s^{pu\ell}}{p} \\ \frac{(a_{2}-rpu\ell\kappa a_{1})s^{pu\ell}}{p} \end{pmatrix}.
	\end{equation}
	It follows from \eqref{4.4}, \( \gcd(s, p) = 1 \), and \( \frac{(a_{2} - rpu\ell \kappa a_{1}) s^{pu\ell}}{p} \in \frac{a_{2} s^{pu\ell}}{p} + \mathbb{Z} \) that \( \Lambda \setminus \{ \bm{0} \} \subset \mathcal{Z}(\widehat{\mu}_{M, D}) \).
	For any \( \lambda_1, \lambda_2 \in \Lambda \setminus \{ \bm{0} \} \) with \( \lambda_1 \neq \lambda_2 \), we can write
	\[
	\lambda_1 = \frac{t^{pu\ell_1} \bm{a}}{p} \quad \text{and} \quad \lambda_2 = \frac{t^{pu\ell_2} \bm{a}}{p}
	\]
	for two positive integers \( \ell_1 > \ell_2 \). Similar to \eqref{4.4},
	\[
	\lambda_1 - \lambda_2 = {M^{*}}^{rpu\ell_{2}}\left(t^{pu(\ell_{1}-\ell_{2})}-1\right) \begin{pmatrix} \frac{a_{1}s^{pu\ell_{2}}}{p} \\ \frac{(a_{2}-rpu\ell_{2}\kappa a_{1})s^{pu\ell_{2}}}{p} \end{pmatrix},
	\]
	which implies that \( (\Lambda - \Lambda) \setminus \{ \bm{0} \} \subset \mathcal{Z}(\widehat{\mu}_{M, D}) \). Therefore, \( E_{\Lambda} \) forms an infinite orthogonal set in \( L^2(\mu_{M, D}) \). In summary, we have thus established the proof of the theorem.
\end{proof}
\begin{proof}[\textbf{Proof of Theorem \ref{thm1.8}}]
	(i) We prove the conclusion by contradiction. Suppose, for the sake of argument, that there exists an orthogonal set \( E_{\Lambda} \) such that \( \#\Lambda > p \). Without loss of generality, we assume that \( \#\Lambda = p+1 \), and let \( \Lambda = \{ \lambda_{0}=\bm{0}, \lambda_1, \lambda_2,\cdots, \lambda_p \} \). For any distinct \( i, j \in \{ 0, 1, 2, \cdots, p \} \), \(  \lambda_i - \lambda_j \in \mathcal{Z}(\widehat{\mu}_{M, D}) =\cup_{j=1}^{p-1} \mathcal{H}_{j} \). Let \( c = \kappa \rho^{-1} \) where \( \kappa \notin \mathbb{Q} \). We can express \( \lambda_i \) as
	\begin{equation}\label{4.5}
		\lambda_i = \rho^{-j_i} \begin{pmatrix}
			\frac{l_i a_1}{p} + k_{j_i,1} \\
			\frac{j_i \kappa l_i a_1 + l_i a_2}{p} + j_i \kappa k_{j_i,1} + k_{j_i,2}
		\end{pmatrix},
	\end{equation}
	where \( \{j_i\}_{i=1}^{p} \subset \mathbb{N} \) is a monotonically decreasing sequence, \( k_{j_i} = (k_{j_{i},1}, k_{j_{i},2})^* \in \mathbb{Z}^2 \), \( l_i \in \{1, 2, \cdots, p-1\} \) for \( i = 1, 2, \cdots, p \), and \( \bm{a} = (a_1, a_2)^* \in \mathcal{E}_{p}\). Considering the orthogonality of \( \Lambda \) and its first component, we have \( j_1 \equiv j_2 \equiv \cdots \equiv j_p \pmod{r} \), which implies that there exists \( 0 \leq b \leq r - 1 \) such that
	\[
	j_i = h_i r + b,\quad i = 1, 2, \cdots, p.
	\]
	Clearly, \( \{h_i\}_{i=1}^{p}  \) is a monotonically decreasing sequence. We will divide the proof into two cases as follows.
	
	\textbf{Case I.} \( s, t \notin p\mathbb{Z} \). By the pigeonhole principle, and since \(s, t, l_i a_1 + pk_{j_i,1} \notin p\mathbb{Z} \) for \( i = 1, 2, \cdots, p \), it is clear that at least two of the following set
	\[
	\left\{ \left( l_{i}a_1 + pk_{j_i,1} \right) s^{h_1 - h_i} t^{h_i} \right\}_{i=1}^p
	\]
	are congruent modulo \( p \). We can assume, without loss of generality, that
	\begin{equation}\label{4.6}
		\left( l_{1}a_1 + pk_{j_1,1}\right)q^{h_1} \equiv 1 \pmod{p}\;\; \text{and}\;\;\left( l_{2}a_1 + pk_{j_2,1} \right) s^{h_1 - h_2} t^{h_2} \equiv 1 \pmod{p}.
	\end{equation}
	Since \( \lambda_1 - \lambda_2 \in \cup_{j=1}^{p-1} \mathcal{H}_j \), there exist \( j_{p+1} \in \mathbb{N} \), \( k_{j_{p+1}} = (k_{j_{p+1},1}, k_{j_{p+1},2})^* \in \mathbb{Z}^2 \), and \( l_{p+1} \in \{1, 2, \cdots, p-1\} \) such that
	\[
	\lambda_1 - \lambda_2 = \rho^{-j_{p+1}} \begin{pmatrix}
		\frac{l_{p+1} a_1}{p} + k_{j_{p+1},1} \\
		\frac{j_{p+1} \kappa l_{p+1} a_1 + l_{p+1} a_2}{p} + j_{p+1} \kappa k_{j_{p+1},1} + k_{j_{p+1},2}
	\end{pmatrix}.
	\]
	Using Lemma \ref{lemma4.3}, \( j_1 \equiv j_2 \equiv j_{p+1} \pmod{r} \). Let \( j_{p+1} = h_{p+1}r + b \). Then, the equation
	\[
	\rho^{-j_1} \left( l_1 a_1 + p k_{j_1,1} \right) - \rho^{-j_2} \left( l_2 a_1 + p k_{j_2,1} \right) = \rho^{-j_{p+1}} \left( l_{p+1} a_1 + p k_{j_{p+1},1} \right)
	\]
	implies that
	\[
	\left( \frac{t}{s} \right)^{h_1} \left( l_1 a_1 + p k_{j_1,1} \right) - \left( \frac{t}{s} \right)^{h_2} \left( l_2 a_1 + p k_{j_2,1} \right) = \left( \frac{t}{s} \right)^{h_{p+1}} \left( l_{p+1} a_1 +p k_{j_{p+1},1} \right).
	\]
	One can easily derive that
	\[
	s^{h_{p+1}} \left( t^{h_1} \left( l_1 a_1 + p k_{j_1,1} \right) - t^{h_2} s^{h_1 - h_2} \left( l_2 a_1 + p k_{j_2,1} \right) \right) = s^{h_{1}}t^{h_{p+1}} \left( l_{p+1} a_1 + p k_{j_{p+1},1} \right).
	\]
	The left-hand side of the above equation belongs to \( p\mathbb{Z} \) as indicated in \eqref{4.6}, while the right-hand side does not. This leads to a contradiction.
	
	\textbf{Case II.}
	$s \in p\mathbb{Z}, \ t \notin p\mathbb{Z} $ or $ s \notin p\mathbb{Z}, \ t \in p\mathbb{Z}.$ By the pigeonhole principle, at least two of \( \lambda_{1}, \lambda_2, \cdots,\lambda_p \) must fall into the same \( \mathcal{H}_i \). We may assume that \( \lambda_1, \lambda_2 \in \mathcal{H}_1 \), and similarly to \eqref{4.5}, \( (a_1 + pk_{j_1,1}) \equiv (a_1 + pk_{j_2,1}) \pmod p \). Since \( \lambda_1 - \lambda_2 \in \cup_{j=1}^{p-1} \mathcal{H}_j \), there exist \( j_{p+2} \in \mathbb{N} \), \( k_{j_{p+2}} = (k_{j_{p+2},1}, k_{j_{p+2},2})^* \in \mathbb{Z}^2 \), and \( l_{p+2} \in \{1, 2, \cdots, p-1\} \) such that
	\begin{equation}\label{4.7}
		\lambda_1 - \lambda_2 = \rho^{-j_{p+2}} \begin{pmatrix}
			\frac{l_{p+2} a_1}{p} + k_{j_{p+2},1} \\
			\frac{j_{p+2} \kappa l_{p+2} a_1 + l_{p+2} a_2}{p} + j_{p+2} \kappa k_{j_{p+2},1} + k_{j_{p+2},2}
		\end{pmatrix}.
	\end{equation}
	From \eqref{4.5} and \eqref{4.7}, it follows that
	\begin{equation}\label{4.8}
		\rho^{-j_{1}}\left(a_{1}+pk_{j_{1},1}\right)-\rho^{-j_{2}}\left(a_{1}+pk_{j_{2},1}\right) = \rho^{-j_{p+2}}\left(l_{p+2}a_{1}+pk_{j_{p+2},1}\right)
	\end{equation}
	and
	\begin{align}\label{4.9}
		\nonumber&\rho^{-j_{1}}\left(j_{1}\kappa a_{1}+a_{2}+p\left(j_{1}\kappa k_{j_{1},1}+k_{j_{1},2}\right)\right)-\rho^{-j_{2}}\left(j_{2}\kappa a_{1}+a_{2}+p\left(j_{2}\kappa k_{j_{2},1}+k_{j_{2},2}\right)\right)\\
		&=\rho^{-j_{p+2}}\left(j_{p+2}\kappa l_{p+2}a_{1}+l_{p+2}a_{2}+p\left(j_{p+2}\kappa k_{j_{p+2},1}+k_{j_{p+2},2}\right)\right).
	\end{align}
	Given that either \( s \) or \( t \) is an element of \( p\mathbb{Z} \), it follows from Lemma \ref{lemma4.3} that \( j_1 \equiv j_2 \equiv j_{p+2} \pmod{r} \), with at least two of the indices \( j_1, j_2, j_{p+2} \) being identical. Note that \( ( a_1 + p k_{j_1,1}) \equiv (a_1 + p k_{j_2,1}) \pmod{p} \), and \( l_{p+2} a_1 + p k_{j_{p+2},1} \notin p\mathbb{Z} \). It follows from \eqref{4.8} that \( j_1 \), \( j_2 \), and \( j_{p+2} \) cannot all be equal. If \( j_1 = j_2 \), then \( j_1 \neq j_{p+2} \). Since \( j_1 \equiv j_{p+2} \pmod{r} \), there exists \( h_{p+2} \in \mathbb{Z} \setminus \{ 0 \} \) such that \( j_1 - j_{p+2} = h_{p+2} r \). Using \( \rho^{-1} =( \frac{t}{s})^{\frac{1}{r}} \) in conjunction with \eqref{4.8} and \eqref{4.9}, we arrive at the conclusion that
	\[
	c\rho=\kappa=\frac{\left(a_{2}+pk_{j_{1},2}\right)-\rho^{j_{1}-j_{p+2}}\left(l_{p+2}a_{2}+pk_{j_{p+2},2}+a_{2}+pk_{j_{2},2}\right)}{(j_{p+2}-j_{1})\left(a_{1}+pk_{j_{1},1}\right)} \in \mathbb{Q}.
	\]
	This leads to a contradiction since \( \kappa \notin \mathbb{Q} \). Similarly, if \( j_1 = j_{p+2} \) or \( j_2 = j_{p+2} \), we can also derive a contradiction. Therefore, based on the above reasoning, there can be at most \( p \) mutually orthogonal exponential functions in \( L^2(\mu_{M, D}) \).
	
	Let \( k \in \mathbb{N} \), and let
	\[
	\Lambda = {M^*}^{k}\left\{\bm{0},\frac{\bm{a}}{p},
	\frac{2\bm{a}}{p},
	\cdots,\frac{(p-1)\bm{a}}{p}
	\right\}.
	\]
	It is straightforward to verify that \( (\Lambda - \Lambda) \setminus \{\bm{0}\} \subseteq \mathcal{Z}(\widehat{\mu}_{M,D}) \). Therefore, having \( p \) elements is optimal, and the result follows.
	
	(ii) Using a proof similar to that in \textbf{Case I} of $(i)$, we can easily complete the proof of $(a)$, and will omit the details here. We will now prove $(b)$. Since \( \kappa \in \mathbb{Q} \) and \( s \in p\mathbb{Z} \), we express \( \kappa = \frac{v}{u} \) for some \( u \in \mathbb{N} \), \( v \in \mathbb{Z} \) with \( \gcd(u, v) = 1 \), and let \( s = p^{\ell} s_1 \) where \( \gcd(s_1, p) = 1 \). For any \( N \in \mathbb{N} \) and \( n \leq N \), we define
	\[
	\Lambda_{n}^{N} =
	\bigcup_{n=1}^{N}\left\{
	 {M^{*}}^{purn}\frac{t^{pu(N-n)}s_{1}^{pu(n-1)}}{p} \bm{a} \right\}.
	\]
	Note that \( t, s_1 \notin p\mathbb{Z} \). It follows easily that
	\begin{equation}\label{4.10}
		\Lambda_n^N \setminus \{\bm{0}\} \subset \mathcal{Z}(\widehat{\mu}_{M, D}).
	\end{equation}
	For any \( \lambda_1 \neq \lambda_2 \in \Lambda_{n}^{N} \), we define
	\[
	\lambda_1 = {M^{*}}^{purn_{1}}\frac{t^{pu(N-n_{1})}s_{1}^{pu(n_{1}-1)}}{p} \bm{a} \quad \text{and} \quad
	\lambda_2 = {M^{*}}^{purn_{2}}\frac{t^{pu(N-n_{2})}s_{1}^{pu(n_{2}-1)}}{p} \bm{a},
	\]
	where \( n_{1} > n_{2} \). From \( s = p^{\ell}s_{1} \) and \( t, s_{1} \notin p\mathbb{Z} \), it follows that
	\begin{align*}
		\lambda_1 - \lambda_2 &= {M^{*}}^{purn_{1}}\frac{t^{pu(N-n_{1})}s_{1}^{pu(n_{1}-1)}}{p} \bm{a}-{M^{*}}^{purn_{1}}{M^{*}}^{pur(n_{2}-n_{1})}\frac{t^{pu(N-n_{2})}s_{1}^{pu(n_{2}-1)}}{p} \bm{a}\\
		&= {M^{*}}^{purn_{1}}\frac{t^{pu(N-n_{1})}s_{1}^{pu(n_{1}-1)}}{p} \left(\bm{a}-p^{pu\ell(n_{1}-n_{2})}\begin{bmatrix} 1 & 0 \\ -vpr(n_{1}-n_{2}) & 1 \end{bmatrix} \bm{a}\right)\\
		&\subset\mathcal{Z}(\widehat{\mu}_{M, D}),
	\end{align*}
	which implies that
	\begin{equation}\label{4.11}
		\left(\Lambda_n^N - \Lambda_n^N\right) \setminus \{\bm{0}\}\subset \mathcal{Z}(\widehat{\mu}_{M, D}).
	\end{equation}
	Combining with \eqref{4.10} and \eqref{4.11}, we can conclude that for any \( N \), \( E_{\Lambda_n^N} \) is the orthogonal set of \( L^2(\mu_{M, D}) \), thus completing the proof.
\end{proof}
Next, we proceed to prove Theorem \ref{thm1.5}. From Proposition \ref{pro 3.4}, it follows that when $\mu_{M, D}$ is a spectral measure, $\rho = \frac{s}{t}$ with $\gcd(s, t) = 1$ and $p \mid t$. Furthermore, we can establish the following proposition.
\begin{pro}\label{pro7.1}
	Let \( \mu_{M,D} \) be given by \eqref{1.1}, \eqref{1.2} and \eqref{x1.3}, where $\rho_{1}=\rho_{2}=\rho = \frac{s}{t}$ and $c = \kappa \rho^{-1}$ for some $s, t \in \mathbb{N}$ with $\gcd(s, t) = 1$, $p \mid t$, and $\kappa \in \mathbb{Q}$. If $s > 1$, then $\mu_{M, D}$ is not a spectral measure.
\end{pro}
\begin{proof}
The proof of this proposition follows by a straightforward adaptation of the proof of Proposition 3.2 in \cite{CLY}, and we omit the details here.
\end{proof}
\begin{lemma}\label{lemma 4.1}
Let $\mu_{M,D}$ be defined by \eqref{1.1}, \eqref{1.2} and \eqref{x1.3}, assume that $\rho_{1} = \rho_{2} = \rho^{-1} = t \in p\mathbb{Z}$, $c = a't$, and $a' = \frac{v}{u} \in \left\{ \frac{k}{l} : \gcd(k, l) = 1 \text{ and } k \in \mathbb{Z} \right\}$, where $\gcd(v, u) = 1$ and $u = p^{\ell_1} u' \in p^{\ell_1} (\mathbb{Z} \setminus p\mathbb{Z})$ for some $\ell_1 \in \mathbb N$. If \( p^{\ell_1 + 1} \mid t \), then \( \mu_{M, D} \) is a spectral measure.
\end{lemma}
\begin{proof}
	Let \( P = \begin{bmatrix} u' & 0 \\ 0 & 1 \end{bmatrix} \). Then, we can write
	$M_{1}=PMP^{-1}=\begin{bmatrix} t & \frac{vt}{p^{\ell_{1}}} \\ 0 & t \end{bmatrix}\;\text{and}\; D_{1}=PD.$
	The spectral properties of \( \mu_{M, D} \) and \( \mu_{M_1, D_1} \) are identical, and
\[ \mathcal{Z}(m_{D_{1}})={P^{*}}^{-1}\mathcal{Z}(m_{D})={P^{*}}^{-1}\bigcup_{j=1}^{p-1}\left(\frac{j\bm{a}+p\mathbb{Z}^{2}}{p}\right), \]
	where $\bm{a}=(a_{1},a_{2})^{*}.$
	Define
	\[ L_{1} := \left\{
	\begin{pmatrix} 0 \\ 0 \end{pmatrix},
	\begin{pmatrix} \frac{ta_{1}}{p} \\ \frac{vta_{1} + p^{\ell_{1}}u'ta_{2}}{p^{\ell_{1}+1}} \end{pmatrix},
	2\begin{pmatrix} \frac{ta_{1}}{p} \\ \frac{vta_{1} + p^{\ell_{1}}u'ta_{2}}{p^{\ell_{1}+1}} \end{pmatrix}, \cdots,
	(p-1)\begin{pmatrix} \frac{ta_{1}}{p} \\ \frac{vta_{1} + p^{\ell_{1}}u'ta_{2}}{p^{\ell_{1}+1}} \end{pmatrix}
	\right\}. \]
	Since \( p^{\ell_1 + 1} \mid t \), it follows that \( L_1 \subset \mathbb{Z}^2 \). By calculation, we find that \( (M_1, D_1, L_1) \) forms a Hadamard triple. Therefore, by Lemma \ref{lem2.3}, it follows that \( \mu_{M_1, D_1} \) is a spectral measure. This completes the proof.
\end{proof}
\begin{lemma}\label{lemma 4.2}
	Suppose that \( \rho_{1}^{-1} =\rho_{2}^{-1} = t \in p\mathbb{Z} \) and \( \lambda_j \in \mathcal{L}_{i_j} \) for \( j = 1, 2 \). If \( \lambda_2 - \lambda_1 \in \mathcal{Z}(\widehat{\mu}_{M,D}) \), then \( \lambda_2 - \lambda_1 \in \mathcal{L}_i:=\mathcal{Z}(\widehat{\delta}_{M^{-i}D}) \), where \( i = \min\{i_1, i_2\} \) if \( i_1 \neq i_2 \), and \( i \geq i_1 \) if \( i_1 = i_2 \).
\end{lemma}
\begin{proof}
	Based on the structure of $\mathcal{L}_{j}$, we can assume
	\[
	\lambda_j = \frac{t^{i_{j}}}{p} {\left({\begin{array}{*{20}{c}}
				s_{j,1}+pk_{j,1}\\
				i_{j}a'\left(s_{j,1}+pk_{j,1}\right)+s_{j,2}+pk_{j,2}\\
		\end{array}}\right)} \in\mathcal{L}_{i_j}
	\]
	for some \( k_{j}=(k_{j,1}, k_{j,2})^{*},s_{j}=(s_{j,1},s_{j,2})^* \in \mathbb{Z}^2 \). Then, \[
	\lambda_2 - \lambda_1 = \frac{t^{i_{1}}}{p} {\left({\begin{array}{*{20}{c}}
				t^{i_{2}-i_{1}}\left(s_{2,1}+pk_{2,1}\right)-s_{1,1}-pk_{1,1}\\
				t^{i_{2}-i_{1}}\left(i_{j}a'(s_{2,1}+pk_{2,1})+s_{2,2}+pk_{2,2}\right)-i_{j}a'(s_{1,1}+pk_{1,1})-s_{1,2}-pk_{1,2}\\
		\end{array}}\right)} .
	\]
	If \( i_1 = i_2 \) and \( i < i_1 \), the first component of the above vector is given by
	\[
	\frac{t^{i_1}}{p}\left(s_{2,1} + pk_{2,1} - s_{1,1} - pk_{1,1}\right) = \frac{t^i}{p} \left( s_{3,1} +p k_{3,1} \right) \; \text{for some} \; s_{3,1} \in \{1, \cdots, p-1\},
	\]
	then
	\[
	t^{i_1 - i}\left(s_{2,1} + pk_{2,1} - s_{1,1} - pk_{1,1}\right) =  s_{3,1} + pk_{3,1}.
	\]
	Since \( p \mid t \), the left-hand side is divisible by \( p \), while the right-hand side belongs to \( \mathbb{Z} \setminus p\mathbb{Z} \), leading to a contradiction. Therefore, we conclude that \( i \geq i_1 \). Similarly, we assume that \( i_2 > i_1 \), so \( \lambda_2 - \lambda_1 \in \mathcal{L}_i \) for some \( i \geq i_1 \). If \( i > i_1 \), then
	\[
	t^{i_2 - i_1} \left( s_{2,1} + pk_{2,1} \right) - s_{1,1} - pk_{1,1}= t^{i - i_1} \left( s_{3,1} + pk_{3,1} \right) \in p\mathbb{Z},
	\]
	which is a contradiction. Therefore, we must have \( i = i_1 \).
\end{proof}
Before proving Theorem \ref{thm1.5}, we define $\mathcal{L}_{j}=\mathcal{Z}(\widehat{\delta}_{M^{-j}D})$ in the same manner as in Lemma \ref{lemma 4.2}. Thus, $\mathcal{Z}(\widehat{\mu}_{M, D})=\cup_{j=1}^{\infty}\mathcal{L}_{j}$.
\begin{proof}[\textbf{Proof of Theorem \ref{thm1.5}}]
	The sufficiency follows directly from Lemma  \ref{lemma 4.1}, and we now need to consider the necessity. If \( \mu_{M, D} \) is a spectral measure, then by Proposition \ref{pro7.1}, \( \rho^{-1} = t \in p^{\ell_1} (\mathbb{Z} \setminus p\mathbb{Z}) \), where \( \ell_1 \geq 1 \). If the necessity does not hold, then either \( a' \notin \mathbb{Q} \) or \( a' = \frac{v}{u} \), where \( u \in p^{\ell_1} \mathbb{Z} \) and \( \gcd(v, u) = 1 \). Let \( \Lambda \) be the spectrum of \( \mu_{M,D} \). From Lemma \ref{lemma 2.2}, it follows that the intersection of \( (\Lambda - \Lambda) \) and \( \mathcal{L}_2 \) is non-empty, i.e., \( (\Lambda - \Lambda) \cap \mathcal{L}_2 \neq \emptyset \). Without loss of generality, we assume that \( \bm{0} \in \Lambda \) and \( \Lambda \cap \mathcal{L}_2 \neq \emptyset \). Lemma \ref{lemma 4.2} then imply that \( \Lambda \cap \mathcal{L}_1 \neq \emptyset \). Furthermore, since \( (\Lambda - \Lambda) \setminus \{\bm{0}\} \subset \cup_{j=1}^{\infty} \mathcal{L}_j \), it follows that there exist \( \lambda_1 \in \mathcal{L}_1 \cap \Lambda \) and \( \lambda_2 \in \mathcal{L}_2 \cap \Lambda \) such that \( \lambda_2 - \lambda_1 \in \mathcal{L}_1 \). Denote
	\[
	\lambda_i = \frac{t^{i}}{p} {\left({\begin{array}{*{20}{c}}
				s_{i,1}+pk_{i,1}\\
				ia'\left(s_{i,1}+pk_{i,1}\right)+s_{i,2}+pk_{i,2}\\
		\end{array}}\right)}
	\]
	for some \( k_{i}=(k_{i,1}, k_{i,2})^{*},s_{i}=(s_{i,1}, s_{i,2})^{*} \in \mathbb{Z}^{2} \), $i=1,2$. Then, \[
	\lambda_2 - \lambda_1 = \frac{t}{p} {\left({\begin{array}{*{20}{c}}
				t\left(s_{2,1}+pk_{2,1}\right)-\left(s_{1,1}+pk_{1,1}\right)\\
				2ta'\left(s_{2,1}+pk_{2,1}\right)+t\left(s_{2,2}+pk_{2,2}\right)-\left(a'\left(s_{1,1}+pk_{1,1}\right)+s_{1,2}+pk_{1,2}\right)\\
		\end{array}}\right)}.
	\]
	By examining the second component of the above expression and the structure of \( \mathcal{L}_1 \), we conclude that
	\begin{equation*}
		t\frac{v}{u}\left(s_{2,1}+pk_{2,1}\right)+t\left(s_{2,2}+pk_{2,2}\right)=s_{3,2}+s_{1,2}+pk_{3,2}+pk_{1,2}
	\end{equation*}
	for some $s_{3,2},k_{3,2}\in\mathbb{Z}. $ However, since \( s_{3,2}+s_{1,2} \equiv 0 \pmod{p} \), it follows that \( t \frac{v}{u}( s_{2,1} +p k_{2,1}) \in p\mathbb{Z} \). Thus, \( p^{\ell_1+1} \mid t \), leading to a contradiction. Therefore, the proof of necessity is complete.
\end{proof}

\section{ The Case $\rho_{1}\neq\rho_2$ }

In this section, we focus on proving Theorems \ref{thm1.3} and \ref{thm1.4}, which investigate the spectral properties of the measure $\mu_{M, D}$ defined by conditions \eqref{1.1}, \eqref{1.2}, \eqref{x1.3} and \eqref{eq1.3} when $\rho_{1}\neq\rho_2$. As established by Theorem \ref{thm1.2}, the spectrality of \(\mu_{M, D}\) necessarily implies that \(c''\) must be a rational number, i.e., \(c'' \in \mathbb{Q}\). Recall that  \begin{equation*}
	E_{\bm{a}}=\left\{ \frac{h}{l} : la_{2}-ha_{1} \in \mathbb{Z}\backslash p\mathbb{Z}, \gcd(h, l) = 1 \right\} \cup \{0\}
\end{equation*}
and
\begin{equation*}
	D = \left\{
	\begin{pmatrix} 0 \\ 0 \end{pmatrix},
	\begin{pmatrix} d_{1,1} \\ 0 \end{pmatrix}, \cdots,
	\begin{pmatrix} d_{p-1,1} \\ d_{p-1,2} \end{pmatrix}
	\right\}.
\end{equation*}
Let $R'=\begin{bmatrix}
	1 & c'' \\
	0 & 1
\end{bmatrix}$, $\widetilde{D}=R'D$. Then
\begin{equation*}
		\widetilde{M} := R'MR'^{-1} = \begin{bmatrix} \rho_1^{-1} & 0 \\ 0 & \rho_2^{-1} \end{bmatrix}, \quad
		\mathcal{Z}(\widehat{\delta}_{\widetilde{D}}) = (R^{'*})^{-1} \bigcup_{j=1}^{p-1} \left( \frac{j\bm{a} + p\mathbb{Z}^2}{p} \right).
\end{equation*}
Consequently, \( \mu_{M, D} \) and \( \mu_{\widetilde{M}, \widetilde{D}} \) share the same spectral properties. Define
\begin{equation*}
	\widetilde{\mu}_{n}=\delta_{{\widetilde{M}^{-1} \widetilde{D}}} *\delta_{{\widetilde{M}^{-2} \widetilde{D}}} * \cdots*\delta_{{\widetilde{M}^{-n} \widetilde{D}}}, \quad
\widetilde{\mu}_{>n}=\delta_{{\widetilde{M}^{-(n+1)} \widetilde{D}}} *\delta_{{\widetilde{M}^{-(n+2)} \widetilde{D}}} * \cdots,
\end{equation*}
and
\(\widetilde{\mathcal{Z}}_j = {\widetilde{M}}^{*^j} \mathcal{Z}(m_{\widetilde{D}}).\)
\begin{proof}[\textbf{Proof of Theorem \ref{thm1.3}}]
	The necessity is directly established by Lemma \ref{lemma 3.7}. We now prove the sufficiency.
	Let $\rho_1^{-1} = t \in p\mathbb{Z}$, $c'' = \frac{c_1}{c_2} \in \mathbb{Q}\backslash E_{\bm{a}}$ with $\gcd(c_{1},c_{2})=1$. Define
	\[
	\mathcal{B} = \left\{ c_2 d_{0,1} + c_1 d_{0,2}, \cdots, c_2 d_{p-1,1} + c_1 d_{p-1,2} \right\},
	\;
\text{	and }
\;
	\mathcal{B}_n = \mathcal{B} + q \mathcal{B} + \cdots + q^{n-1} \mathcal{B}.
	\]
From Lemma \ref{lemma 2.9}, we know that \(\mathcal{B} = \{0, 1, \cdots, p-1\} \pmod{p}.\)	Let
\[  L = \left\{ 0, \frac{t}{p}, \frac{2t}{p}, \cdots, \frac{(p-1)t}{p} \right\} \subset \mathbb{Z}. \]
Then the triplet $(t, \mathcal{B}, L)$ forms a Hadamard triple. By \cite[Theorem 5.4]{DH1}, we know the associated self-similar measure \( \mu_{t, \mathcal{B}} \) is a spectral with a spectrum in $\mathbb Z$. Denote \( L_n = L + tL + t^2 L + \cdots + t^{n-1} L \) and
	\[
	\Lambda = \bigcup_{k=1}^{\infty} \Lambda_k \quad \text{with} \quad
	\Lambda_k = J_{n_1} + t^{\ell_1} J_{n_2} + \cdots + t^{\ell_{k-1}} J_{n_k},
	\]
	where \( 0 \in J_n \equiv L_n \pmod{t^n} \) and \( \ell_k = n_1 + n_2 + \cdots + n_k \).
By \cite[Proposition 4.6]{DH1}, we obtain a spectrum \( \Lambda \) of \( \mu_{t, \mathcal{B}} \)  and
	\[
	\delta(\Lambda):= \inf_{k \geq 1} \inf_{\lambda \in \Lambda_k} \left| \widehat{\mu}_{t, \mathcal{B}} \left( t^{-\ell_k} \lambda \right) \right|^2 > 0.
	\]
	This shows that
	\begin{equation*}
		\prod_{j=1}^{\infty}\frac{1}{p^{2}}	\left|\sum_{l=0}^{p-1}e^{-2\pi i(c_{2}d_{l,1}+c_{1}d_{l,2})\lambda t^{-\ell_{k}-j}}\right|^{2}\ge \delta(\Lambda)>0,\;\lambda\in\Lambda_{k}.
	\end{equation*}
	Define \( \Lambda'_k = \{(c_2 \lambda, 0)^* : \lambda \in \Lambda_k \} \) and \( \Lambda' = \cup_{k=1}^{\infty} \Lambda'_k \). For any distinct \( \lambda_1', \lambda_2' \in \Lambda'_k \), since \( \lambda_1 - \lambda_2 \in \frac{1}{p} t^\ell (\mathbb{Z} \setminus p\mathbb{Z}) \) with \( \ell \leq \ell_k \), it can be concluded that
	\[
	\lambda_{1}' - \lambda_{2}' = (c_2 (\lambda_{1} - \lambda_{2}), 0)^* \in \left\{ \left( \frac{c_2}{p} t^\ell (s + pk), 0 \right)^* :  k \in \mathbb{Z}, s = 1,2,\cdots,p-1\right\} .
	\]
	Note that
	\begin{align*}
		\begin{bmatrix}
			t^{\ell} & 0 \\
			0 & \rho_2^{-\ell}
		\end{bmatrix}\mathcal{Z}(m_{R'D})&=\begin{bmatrix}
			t^{\ell} & 0 \\
			0 & \rho_2^{-\ell}
		\end{bmatrix}\begin{bmatrix}
			1 & 0 \\
			-c'' & 1
		\end{bmatrix}\bigcup_{j=1}^{p-1}\left(\frac{j\bm{a}+p\mathbb{Z}^{2}}{p}\right)\\
		&=\bigcup_{j=1}^{p-1}{\left({\begin{array}{*{20}{c}}
					t^{\ell}\left(\frac{ja_{1}}{p}+z_{1}\right)\\
					\rho_2^{-\ell}\left(\frac{j(-c''a_{1}+a_{2})}{p}-c''z_{1}+z_{2}\right)
			\end{array}}\right)}
	\end{align*}
	for any $z=(z_{1},z_{2})^{*}\subset\mathbb{Z}^{2}$. Given that \( \gcd(c_2, p) = 1 \) and \( c'' \in \mathbb{Q}\backslash E_{\bm{a}} \), it follows that
	$ \lambda_{1}' - \lambda_{2}' \subset\mathcal{Z}(\widehat{\widetilde{\mu}}_{\ell_{k}}).$
	Therefore, it is not hard to conclude that \( \Lambda'_k \) is a spectrum of \( \widetilde{\mu}_{\ell_{k}} \). Moreover, since for any \( \lambda'= (c_2 \lambda, 0)^* \in \Lambda'_k \), we have
	\begin{equation*}
		\left|\widehat{\widetilde{\mu}}_{>\ell_k}(\lambda')\right|^2=\prod_{j=1}^{\infty}\frac{1}{p^{2}}	\left|\sum_{l=0}^{p-1}e^{-2\pi i\left(c_{2}d_{l,1}+c_{1}d_{l,2}\right)\lambda t^{-\ell_{k}-j}}\right|^{2}\ge \delta(\Lambda)>0,\;\lambda\in\Lambda_{k},
	\end{equation*} the equicontinuity of the sequence \( \{\widehat{\widetilde{\mu}}_{>\ell_k}\}_{k=1}^{\infty} \) implies that
	\[
	\left|\widehat{\widetilde{\mu}}_{>\ell_k}(\lambda' + \xi)\right|^2 \geq \frac{1}{2} \delta(\Lambda), \quad \xi \in [0, \gamma]^2 \text{ for some } \gamma > 0.
	\]
	Thus, \( \Lambda' \) is the spectrum of \( \mu_{\widetilde{M}, \widetilde{D}} \), as established in Lemma \ref{lemma 2.7} (ii).
\end{proof}
Next, we consider the case \( c''\in E_{\bm{a}} \).
\begin{pro}\label{pro 5.1}
	If \( \mu_{\widetilde{M}, \widetilde{D}} \) is a spectral measure and \( c'' \in E_{\bm{a}} \), then \( \rho_1^{-1}, \rho_2^{-1} \in p \mathbb{Z} \).
\end{pro}
\begin{proof}
	$\rho_1^{-1} \in p \mathbb{Z}$ follows from Lemma \ref{lemma 3.7}. Let $c''=\frac{c_{1}}{c_{2}}\in  E_{\bm{a}}$ with $\gcd(c_{1},c_{2})=1.$ Since
	\begin{align*}
		\mathcal{Z}(\widehat{\mu}_{\widetilde{M}, \widetilde{D}})&=\bigcup_{i=1}^{\infty}\bigcup_{j=1}^{p-1}{\left({\begin{array}{*{20}{c}}
					\rho_1^{-i}\left(\frac{ja_{1}}{p}+z_{1}\right)\\
					\rho_2^{-i}\left(\frac{j(-c''a_{1}+a_{2})}{p}-c''z_{1}+z_{2}\right)
			\end{array}}\right)}\\
		&\subset\bigcup_{i=1}^{\infty}\left\{{\frac{1}{c_{2}}\left({\begin{array}{*{20}{c}}
					x\\
					\rho_2^{-i}\left(\frac{\ell}{p}+\mathbb{Z}\right)
			\end{array}}\right)}:x\in\mathbb{R},\ell=1,2,\cdots,p-1\right\},
	\end{align*}
	similar to the proof of Proposition \ref{pro 3.4}, we can conclude that \( \rho_2^{-1} = \frac{t}{s} \) for some \( t \in p \mathbb{Z} \) with \( \gcd(s, t) = 1 \). If \( s > 1 \), let \( \xi = (\xi_1, \xi_2)^* \in \mathbb{R}^2 \) and
	\[
	f(\xi_1, \xi_2) = |m_{\widetilde{D}}(\xi)| = \frac{1}{p} \left| \sum_{j=0}^{p-1} e^{-2\pi i \left(\xi_1\left(d_{j,1}+c''d_{j,2}\right)+\xi_{2}d_{j,2}\right)}\right|.
	\]
	Clearly, $f(\xi_1 + c_2, \xi_2 + 1) = f(\xi_1, \xi_2)$ and
	\begin{align*}
		\left\{ (\xi_1, \xi_2)^* : f(\xi_1, \xi_2) = 1 \right\} &= \left\{ (\xi_1, \xi_2)^* :  \xi_1\left(d_{i,1}+c''d_{i,2}\right)+\xi_{2}d_{i,2} \in \mathbb{Z},i\in\{0,1,\cdots,p-1\} \right\}\\
		&\subset \left\{ \left(\xi_1, k (c_2d_{i,2})^{-1}\right)^* : \xi_1 \in \mathbb{R}, k \in \mathbb{Z} \right\}.
	\end{align*}
	From Lemma \ref{lemma 2.8},  for any $n\geq 1$ and \( \xi = (\xi_1, \xi_2)^* \) with \( |\rho_2^{-n} \xi_2| > 1 \), there exist positive constants \( \alpha, \beta \) such that
	\[
	\left|\widehat{\widetilde{\mu}}_{>n}(\xi)\right| = \prod_{j=1}^{\infty} \left|m_{\widetilde{D}}\left(\widetilde{M}^{*^{-(j+n)}} \xi\right)\right| \leq \alpha \left( \ln | \rho_2^{-n} \xi_2 | \right)^{-\beta}.
	\]
Thus, by Proposition \ref{pro 3.6}, \( \mu_{\widetilde{M}, \widetilde{D} }\) cannot be a spectral measure, leading to a contradiction. Hence, \( s = 1 \), which implies \( \rho_2^{-1} \in p\mathbb{Z} \).
\end{proof}
\begin{pro}\label{pro 5.2}
	Suppose that \( \rho_{1}^{-1}, \rho_{2}^{-1} \in p\mathbb{Z} \) and \( c'' = \frac{c_1}{c_2} \in  E_{\bm{a}} \), where \( \gcd(c_1, c_2) = 1 \) and \( c_2 \in p^{\ell} (\mathbb{Z} \setminus p\mathbb{Z}) \) for some \( \ell \in \mathbb{N} \). If \( p^{\ell+1} \nmid (\rho_{1}^{-1} - \rho_{2}^{-1}) \), then \( \mu_{\widetilde{M}, \widetilde{D}} \) is not a spectral measure.
\end{pro}
\begin{proof}
Assuming that \( \mu_{\widetilde{M}, \widetilde{D}} \) is spectral measure with a spectrum \(\bm{0}\in\Lambda \). From \( \rho_1^{-1}, \rho_2^{-1} \in p\mathbb{Z} \) and Lemma \ref{lemma 3.8}, it follows that \( \Lambda \cap \widetilde{\mathcal{Z}}_j \neq \emptyset \) for \( j = 1, 2 \).
	Denote
	\[
	\lambda_j = \frac{1}{p} {\left({\begin{array}{*{20}{c}}
				\rho_{1}^{-j}\left(i_{j}a_{1}+pk_{j,1}\right)\\
				\rho_{2}^{-j}\left(	-c''\left(i_{j}a_{1}+pk_{j,1}\right)+i_{j}a_{2}+pk_{j,2}\right)\\
		\end{array}}\right)} \subset \Lambda \cap \widetilde{\mathcal{Z}}_j
	\]
	for some \( k=(k_{j,1}, k_{j,2})^{*} \in \mathbb{Z}^{2} \), $i_{j}\in\{1,2,\cdots,p-1\}$, $j=1,2$. Then, $\lambda_2 - \lambda_1\in\widetilde{\mathcal{Z}}_{1}$ follows as
	\[
	\frac{1}{p} {\left({\begin{array}{*{20}{c}}
				\rho_1^{-1}\left(\rho_1^{-1} \left(i_{2}a_{1}+pk_{2,1}\right)-\left(i_{1}a_{1}+pk_{1,1}\right)\right)\\
				\rho_2^{-1}\left(\rho_2^{-1}\left(	i_{2}a_{2}+pk_{2,2}-c''\left(i_{2}a_{1}+pk_{2,1}\right)\right)+c''\left(i_{1}a_{1}+pk_{1,1}\right)-i_{1}a_{2}-pk_{1,2}\right)
		\end{array}}\right)}.
	\]
	Furthermore, we get
	\begin{equation*}
		\left(\rho_{1}^{-1} - \rho_{2}^{-1}\right)c''\left(i_{2}a_{1}+pk_{2,1}\right)	=i_{3}a_{2}+pk_{3,2}+i_{1}a_{2}+pk_{1,2}-\rho_{2}^{-1}\left(i_{2}a_{2}+pk_{2,2}\right).
	\end{equation*}
	From the expression $i_{3}a_{2}+pk_{3,2}+i_{1}a_{2}+pk_{1,2}\in p\mathbb{Z} $, it follows that $(\rho_{1}^{-1} - \rho_{2}^{-1})c''(i_{2}a_{1}+pk_{2,1})\in p\mathbb{Z} $. This implies that \( p^{\ell+1} \mid ( \rho_1^{-1} - \rho_2^{-1} ) \), which leads to a contradiction. Therefore, the proof is complete.
\end{proof}
Combining with Proposition \ref{pro 5.1} and Proposition \ref{pro 5.2}, we can derive the necessary part of Theorem \ref{thm1.4}. For the sufficiency, it requires several additional technical lemmas.
For convenience,
we define
\[ P'=\begin{bmatrix}
	c_{2}' & 0 \\
	-\frac{c_{1}}{p^{u}} & -1
\end{bmatrix}, \qquad
G=[0,a_{1}c_{2}'] \times\left[-\frac{|c_{1}a_{1}-c_{2}a_{2}|}{p^{u+1}},\frac{|c_{1}a_{1}-c_{2}a_{2}|}{p^{u+1}}\right],
\]
and the integral periodic zero set of \( P' \) as
\[
Z(\mu) := \left\{ \xi \in G : \widehat{\mu}(\xi + P'k) = 0 \text{ for all } k \in \mathbb{Z}^2 \right\}.
\]

\begin{lemma}\label{lemma 5.3}
	Suppose that
	$\rho_1^{-1} = t_1, \; \rho_2^{-1} = t_2 \in p\mathbb{Z}$
	and
	$c'' = \frac{c_1}{c_2}, \;\text{where } \gcd(c_1, c_2) = 1 \text{ and } c_2 = p^{u} c_2' \in p^{u}(\mathbb{Z} \setminus p\mathbb{Z})$ for some \( u \in \mathbb{N} \). If \[
	\inf_{\xi \in G} \sup_{k \in \mathbb{Z}^2} \left| \widehat{\mu}_{\widetilde{M}, \widetilde{D}}(\xi + P'k) \right| := \epsilon_0 > 0,
	\]
	then \(\mu_{\widetilde{M}, \widetilde{D}}\) is a spectral measure.
\end{lemma}
\begin{proof}
	According to Lemma \ref{lem2.8}, \(\widehat{\mu}_{\widetilde{M}, \widetilde{D}}\) is uniformly continuous. This implies that there exists \(\delta> 0\) such that for any \(\xi \in G\), there exists \(k_\xi \in \mathbb{Z}^2\) satisfying
{\small
\begin{equation}\label{5.3}
		\left| \widehat{\mu}_{\widetilde{M}, \widetilde{D}}(\xi + y + P'k_\xi) \right| \geq
\left| \left| \widehat{\mu}_{\widetilde{M}, \widetilde{D}}(\xi + P'k_\xi) \right| - \left| \widehat{\mu}_{\widetilde{M}, \widetilde{D}}(\xi + y + P'k_\xi) - \widehat{\mu}_{\widetilde{M}, \widetilde{D}}(\xi + P'k_\xi) \right| \right|
\geq \frac{\epsilon_0}{2},
	\end{equation}
}
whenever $\parallel y \parallel\le\delta.$
Construct the sequence \(\{\Lambda_{n_{\ell}}\}\) with \( n_0 = 0 \) and \( n_{\ell+1}>n_{\ell} \) such that
\begin{equation}\label{5.4}
		\inf_{\lambda_\ell \in \Lambda_{n_\ell}} \left| \widehat{\mu}_{\widetilde{M}, \widetilde{D}}\left({\widetilde{M}} ^{- n_\ell} \lambda_\ell\right) \right| \geq \frac{\epsilon_0}{2}.
\end{equation}
For any \( \bm{w}_\ell \in \{0, 1, 2,\cdots,p-1\}^{n_{\ell+1} - n_\ell} := \Xi^{n_{\ell+1} - n_\ell} \), let
\[
	x_{\bm{w}_\ell} = {\widetilde{M}}^{-n_{\ell+1}} {\left({\begin{array}{*{20}{c}}
				\frac{a_{1}c'_{2}}{p}\sum_{i=n_{\ell}+1}^{n_{\ell+1}}t_{1}^{i}w_{i}\\
				-\frac{a_{1}c_{1}-a_{2}c_{2}}{p^{u+1}}\sum_{i=n_{\ell}+1}^{n_{\ell+1}}t_{2}^{i}w_{i}\\
		\end{array}}\right)}
\]
	and
	\[
	\Lambda_{n_{\ell+1}} = \Lambda_{n_\ell} + \left\{ {\widetilde{M}}^{n_{\ell+1}} (x_{\bm{w}_\ell} + P' k _{x_{\bm{w}_\ell}}) : \bm{w}_\ell \in \Xi^{n_{\ell+1} - n_\ell} \right\}.
	\]
Along with \eqref{5.3} and \eqref{5.4}, we have
\begin{equation}\label{5.5}
\sup_{\lambda_\ell \in \Lambda_{n_\ell}} \parallel {\widetilde{M}}^{ - n_{\ell+1}} \lambda_k\parallel <\delta.
\end{equation}

	Furthermore, we will prove that \( \Lambda=\cup_{\ell=1}^{\infty}\Lambda_{n_\ell} \) is the spectrum of \( \mu_{\widetilde{M}, \widetilde{D}} \). For any \( \lambda_1 \neq \lambda_2 \in \Lambda_{n_\ell} \), there exist two distinct sequences \( \{w_1^{(i)}\}_{i=1}^{n_\ell} \),\( \{w_2^{(i)}\}_{i=1}^{n_\ell} \subset \{0, 1, 2, \cdots, p-1\}^{n_\ell} \), such that
	\[
	\lambda_i = \sum_{j=1}^{\ell} {\widetilde{M}}^{n_j} \left( x_{w_i^{(j)}} + P'k_{x_{w_i^{(j)}}} \right),
	\]
where \( k_{x_{w_i^{(j)}}} = ( k_{j_{ i},1}, k_{j_{i},2} )^* \) and \( l = \min \{ j : w_1^{(j)} \neq w_2^{(j)} \} \in (n_s, n_{s+1}] \) for some \( 0 \leq s \leq \ell-1 \). The expression for \( \lambda_1 - \lambda_2 \) can be rewritten as follows:
	\[
	{\left({\begin{array}{*{20}{c}}
				\frac{a_{1}c'_{2}}{p}\sum_{j=l}^{n_{\ell}}t_{1}^{j}\left(w_{1}^{(j)}-w_{2}^{(j)}\right)+\sum_{j=s+1}^{\ell}t_{1}^{n_j}c'_{2}\left(k_{j_{1},1}-k_{j_{2},1}\right)\\
				-\frac{a_{1}c_{1}-a_{2}c_{2}}{p^{u+1}}\sum_{j=l}^{n_{\ell}}t_{2}^{j}\left(w_{1}^{(j)}-w_{2}^{(j)}\right)-\sum_{j=s+1}^{\ell}t_{2}^{n_j}\left(\frac{c_{1}\left(k_{j_{1}, 1}-k_{j_{2}, 1}\right)}{p^{u}}+k_{j_{ 1},2}-k_{j_{2}, 2}\right)
		\end{array}}\right)}.
	\]
	It follows from \( p^{u + 1} \mid (t_1 - t_2) \) that
	\begin{align*}
		m_{\widetilde{D}}\left( {\widetilde{M}}^{ - l} (\lambda_1 - \lambda_2) \right) &= \frac{1}{p}\sum_{j=0}^{p-1} e^{-2\pi i \frac{c'_2\left( w_1^{(l)} - w_2^{(l)} \right)}{p}\left( d_{j,1}a_{1} + d_{j,2}a_{2} \right)  }\\
		&= \frac{1}{p}\sum_{j=0}^{p-1} e^{-2\pi i \frac{c'_2\left( w_1^{(l)} - w_2^{(l)} \right)}{p}\left\langle d_{j}, \bm{a}\right\rangle  }= 0.
	\end{align*}
	Therefore, we conclude that \( \widetilde{\mu}_{n_{\ell}} (\lambda_1 - \lambda_2) = 0 \). Since \( \# \text{spt} \, \widetilde{\mu}_{n_{\ell}} = p^{n_\ell} \), it immediately follows that \( \Lambda_{n_\ell} \) is a spectrum of \( \widetilde{\mu}_{n_{\ell}} \) for all \( \ell \geq 1 \). By Lemma \ref{lem2.8}, the sequence \( \{ \widehat{\widetilde{\mu}}_{>n_\ell} \}_{\ell=1}^{\infty} \) is equicontinuous.
Therefore, for $\lambda \in \Lambda_{n_{\ell+1}} \setminus \Lambda_{n_\ell}, \, \xi \in [0, r]^{2} $ with some $r>0$,  using \eqref{5.3}, we obtain that
	\[
	\left|\widehat{\widetilde{\mu}}_{>n_{\ell+1}} (\lambda + \xi)\right| \geq \left|\widehat{\widetilde{\mu}}_{>n_{\ell+1}} (\lambda)\right| - \left|\widehat{\widetilde{\mu}}_{>n_{\ell+1}} (\lambda) - \widehat{\widetilde{\mu}}_{>n_{\ell+1}} (\lambda + \xi)\right|
	\geq \frac{\epsilon_0}{2}.
	\]
Moreover, by Lemma \ref{lemma 2.7}, \( \Lambda \) is a spectrum of \( \mu_{\widetilde{M}, \widetilde{D}} \).
\end{proof}

\begin{lemma}\label{lemma 5.4}
	Suppose that
	$\rho_1^{-1} = t_1, \; \rho_2^{-1} = t_2 \in p\mathbb{Z}$
	and
	$c'' = \frac{c_1}{c_2}, \;\text{where } \gcd(c_1, c_2) = 1 \text{ and } c_2 = p^{u} c_2' \in p^{u}(\mathbb{Z} \setminus p\mathbb{Z})$ for some \( u \in \mathbb{N} \). If $Z(\mu_{\widetilde{M}, \widetilde{D}}) = \emptyset$, then
	\[
	\inf_{\xi \in G} \sup_{k \in \mathbb{Z}^2} \left| \widehat{\mu}_{\widetilde{M}, \widetilde{D}}(\xi + P'k) \right| > 0.
	\]
\end{lemma}
\begin{proof}
	For every \( \xi \in G \), the condition \( Z(\mu_{\widetilde{M}, \widetilde{D}}) = \emptyset \) implies the existence of an integer vector
	$
	k_{\xi} \in \mathbb{Z}^2
	$
	such that
	\[
	\left| \widehat{\mu}_{\widetilde{M}, \widetilde{D}}(\xi + P' k_{\xi}) \right| \geq \epsilon_\xi > 0.
	\]
	By the continuity of \( \widehat{\mu}_{\widetilde{M}, \widetilde{D}} \), we obtain
	\[
	\left| \widehat{\mu}_{\widetilde{M}, \widetilde{D}}(\xi + y + P' k_{\xi}) \right| \geq \frac{1}{2} \epsilon_\xi > 0 \quad \text{whenever} \quad \| y \| < \delta_\xi.
	\]
	Define
	$B(\xi, \delta_\xi) := \{ x \in \mathbb{R}^2 : \text{dist}(x, \xi) < \delta_\xi \}.$
	Since \( G \) is a compact set and \( G \subset \cup_{\xi \in G} B(\xi, \delta_\xi) \), by the finite covering theorem, there exists a finite collection of balls
	$\{ B(\xi_i, \delta_{\xi_i}) \}_{i=1}^\ell$
	such that
	$\cup_{i=1}^\ell B(\xi_i, \delta_{\xi_i}) \supset G.$
	For any \( \xi \in G \), there exists a \( j \) with \( 1 \leq j \leq \ell \) such that \( \xi \in B(\xi_j, \delta_{\xi_j}) \). Thus,
	\[
	\left| \widehat{\mu}_{\widetilde{M}, \widetilde{D}}(\xi + P' k_{\xi_j}) \right| \geq \frac{1}{2} \epsilon_{\xi_j} \geq \frac{1}{2} \min_{1 \leq i \leq \ell} \{ \epsilon_{\xi_i} \} > 0.
	\]
\end{proof}
According to Lemma \ref{lemma 5.3} and Lemma \ref{lemma 5.4}, for the sufficiency of Theorem \ref{thm1.4}, we only need to prove that \( Z(\mu_{\widetilde{M}, \widetilde{D}}) = \emptyset \). However, before proceeding, we require a key lemma.
\begin{lemma}\label{lemma 5.5}
	Suppose that
	$\rho_1^{-1} = t_1, \; \rho_2^{-1} = t_2 \in p\mathbb{Z}$
	and
	$c'' = \frac{c_1}{c_2}, \;\text{where } \gcd(c_1, c_2) = 1 \text{ and } c_2 = p^{u} c_2' \in p^{u}(\mathbb{Z} \setminus p\mathbb{Z})$ for some \( u \in \mathbb{N} \). For any \( k = (k_1, k_2)^* \in \mathbb{Z}^2 \), if \( \widehat{\mu}_{\widetilde{M}, \widetilde{D}}(\xi + P'k) = 0 \), then there exists \( j_k \) such that $m_{\widetilde{D}}({\widetilde{M}}^{-j_k} (\xi + P'k)) = 0.
	$
\end{lemma}
\begin{proof}
	From the inequalities \( |z| \geq |\text{Re}(z)| \) and \( \cos(2\pi x) \geq \cos(2\pi r) \) for all \( |x| \leq r < \frac{1}{4} \), it follows that
	\[
	\left| \frac{1}{p} \sum_{j=0}^{p-1} e^{2\pi ix_j} \right| \geq \cos(2\pi r),
\quad \text{with} \quad
|x_{j}| \leq r < \frac{1}{4}.
	\]
We can choose an appropriate \( J \) such that for all \( i \in \{0, 1, \cdots, p-1\} \), the following inequality holds:
	\[
	H_i = \rho_1^J (\xi_1 + c_2' k_1)(d_{i,1} + c'' d_{i,2}) + \rho_2^J d_{i,2} \left( \xi_2 - \frac{c_1 k_1}{p^u} - k_2 \right) \leq \frac{1}{16}.
	\]
	Hence,
	\[
	|m_{\widetilde{D}}(\widetilde{M}^{-j}(\xi + P'k))| \geq \cos \left( \frac{1}{4} \max \{ \rho_1^{j-J}, \rho_2^{j-J} \}\right)  \quad \text{for} \quad j \geq J+1.
	\]
Furthermore, we know that
	\[
	\prod_{j=J+1}^\infty |m_{\widetilde{D}}(\widetilde{M}^{-j}(\xi + P'k))| \geq \prod_{j=1}^\infty \cos \left( \frac{1}{4} \max \left\{ \rho_1^j, \rho_2^j \right\} \right) > 0.
	\]
	Together with the assumption that
	\( \widehat{\mu}_{\widetilde{M}, \widetilde{D}}(\xi + P'k) = 0 ,\)
	this implies that
	\( \widehat{\widetilde{\mu}}_{J}(\xi + P'k) = 0 .\)
	Consequently, we have  $m_{\widetilde{D}}({\widetilde{M}}^{-j_k} (\xi + P'k)) = 0$ for some \( j_k\le J \).
\end{proof}
We now proceed to demonstrate that \( Z(\mu_{\widetilde{M},\widetilde{D}}) = \emptyset \).
\begin{lemma}\label{lemma 5.6}
	Suppose that
	$\rho_1^{-1} = t_1, \; \rho_2^{-1} = t_2 \in p\mathbb{Z}$
	and
	$c'' = \frac{c_1}{c_2}, \;\text{where } \gcd(c_1, c_2) = 1 \text{ and } c_2 = p^{u} c_2' \in p^{u}(\mathbb{Z} \setminus p\mathbb{Z})$ for some \( u \in \mathbb{N} \), then	\( Z(\mu_{\widetilde{M},\widetilde{D}}) = \emptyset \).
\end{lemma}
\begin{proof}
	If not, we let \( \xi = (\xi_1, \xi_2)^* \in Z(\mu_{\widetilde{M}, \widetilde{D}}) \). For any \( k = (n, \ell)^* \in \mathbb{Z}^2 \), the definition of \( Z(\mu_{\widetilde{M}, \widetilde{D}}) \) ensures that \( \xi + P'k \in \mathcal{Z}(\widehat{\mu}_{\widetilde{M}, \widetilde{D}}) \). By Lemma \ref{lemma 5.5}, there exists \( j_k \) such that \( \xi + P'k \in\mathcal{Z}( m_{{\widetilde{M}}^{-j_k} \widetilde{D}}) \), i.e.,
	\begin{equation}\label{5.6}
		\xi_1 + c'_2 n = t_{1}^{j_k} \left(\frac{i_{k}a_{1}}{p} + k_1\right)
		\;
		\text{and}
		\;
		\xi_2 - \frac{c_1 n}{p^{u}} - \ell=t_{2}^{j_k}\left(-c''\left(\frac{i_{k}a_{1}}{p}+k_{1}\right)+\left(\frac{i_{k}a_{2}}{p}+k_{2}\right)\right)
	\end{equation}
	for some \(k_1,k_{2} \in \mathbb{Z}\), $i_{k}\in\{1,2,\cdots,p-1\}$. For the convenience of subsequent discussion, we define
	$j_{k} = j_{n,\ell_{n}}, k_i = k_i^{(n,\ell_{n})} $ and $i_{k}=i_{n,\ell_{n}}.$
	Take \( n = 0 \) and \( \ell_0 \in \mathbb{Z} \). By equation \eqref{5.6}, there exist two positive integers \( j_{0,\ell_0} \) and \( k_1^{(0,\ell_0)} \) such that
	\[
	\xi_1 = t_1^{j_{0,\ell_0}} \left( \frac{i_{0,\ell_0} a_1}{p} + k_1^{(0,\ell_0)} \right).
	\]
	Substituting this into equation \eqref{5.6}, we obtain
	\begin{equation}\label{5.7}
		t_1^{j_{0,\ell_0}} \left( \frac{i_{0,\ell_0} a_1}{p} + k_1^{(0,\ell_0)} \right) + c'_2 n = t_1^{j_{n,\ell_n}} \left( \frac{i_{n,\ell_n} a_1}{p} + k_1^{(n,\ell_n)} \right), \quad (n, \ell_n)^* \in \mathbb{Z}^2.
	\end{equation}
	We claim that for any \( (n, \ell_n)^* \in \{0, 1, \cdots, p-1\} \times \mathbb{Z} \), there are exactly \( p-1 \) elements in \( \{j_{n,\ell_n}\}_{n=0}^{p-1} \) that are equal to $1$, and the remaining one is greater than $1$.
	
	Let \( n \in \{1, 2,\cdots,p-1\} \). If \( j_{n,\ell_n} \leq j_{0,\ell_0} \), then \( j_{n,\ell_n} = 1 \) and \( t_1 \in p(\mathbb{Z} \setminus p\mathbb{Z}) \). Otherwise, using equation \eqref{5.7}, we obtain
	\[
	c'_2 n = t_1^{j_{n,\ell_n}} \left( \frac{i_{n,\ell_n} a_1}{p} + k_1^{(n,\ell_n)} \right) - t_1^{j_{0,\ell_0}} \left( \frac{i_{0,\ell_0} a_1}{p} + k_1^{(0,\ell_0)} \right) \in p\mathbb{Z},
	\]
	which leads to a contradiction since \( c'_2 n \in \mathbb{Z} \setminus p\mathbb{Z} \). If \( j_{n,\ell_n} > j_{0,\ell_0} \), then \( j_{n,\ell_n} \geq 2 \) and we have
	\[
	t_1^{j_{0,\ell_0}} \left( \frac{i_{0,\ell_0} a_1}{p} + k_1^{(0,\ell_0)} \right) = t_1^{j_{n,\ell_n}} \left( \frac{i_{n,\ell_n} a_1}{p} + k_1^{(n,\ell_n)} \right) -c'_2 n  \in \mathbb{Z} \setminus p\mathbb{Z}.
	\]
	This implies \( j_{0,\ell_0} = 1 \) and \( t_1 \in p(\mathbb{Z} \setminus p\mathbb{Z}) \). Using \eqref{5.6} again, for any distinct \( n_1, n_2 \in \{1, 2, \cdots, p-1\} \), we obtain
	\[
	(n_1 - n_2) c'_2 \in \left(t_1^{j_{n_1,\ell_{n_1}}} \left( \frac{\{1, 2, \cdots, p-1\}}{p} + \mathbb{Z} \right) - t_1^{j_{n_2,\ell_{n_2}}} \left( \frac{\{1, 2, \cdots, p-1\}}{p} + \mathbb{Z} \right)\right).
	\]
	Together with \( c'_2 \in \mathbb{Z} \setminus p\mathbb{Z} \), this leads to the conclusion that \( \min\{ j_{n_1, \ell_{n_1}}, j_{n_2, \ell_{n_2}} \} = 1 \).
	The arbitrariness of \( n_1 \) and \( n_2 \) tells us that at least \( p-1 \) elements in the set \( \{ j_{n,\ell_n} \}_{n=0}^{p-1} \) must be equal to 1. If they are all equal to 1, let \( hc'_2 = -t_1 \frac{i_{0,\ell_0} a_1}{p} \pmod{p} \) for some \( h \in \{1, 2, \cdots, p-1\} \). Then,
	\[
	t_1 \frac{i_{0,\ell_0} a_1}{p} + h c'_2 = t_1^{j_{h, \ell_{h}}} \left( \frac{i_{h, \ell_{h}} a_1}{p} + k_1^{(h, \ell_{h})} \right) \in \mathbb{Z} \setminus p\mathbb{Z},
	\]
	which leads to a contradiction, and the claim follows.
	
	Building on the previous claim, we assume that \( j_{0, \ell_0} > 1 \) and \( j_{1, \ell_1} = j_{2, \ell_2} = \cdots = j_{p-1, \ell_{p-1}} = 1 \) for all vectors \( (\ell_0, \ell_1, \cdots, \ell_{p-1})^* \in \mathbb{Z}^p \). Let \( n = 1 \) and \( \ell_1 \in \{0, 1,\cdots,p-1\} \). It follows from \eqref{5.6} that
	\[
	\xi_2 - \frac{c_1 }{p^{u}} - \ell_{1}=t_{2}^{j_{1,\ell_{1}}}\left(-c''\left(\frac{i_{1,\ell_{1}}a_{1}}{p}+k_{1}^{({1,\ell_{1}})}\right)+\left(\frac{i_{1,\ell_{1}}a_{2}}{p}+k_{2}^{(1,\ell_{1})}\right)\right).
	\]
	Through simple calculations, we get
	\begin{equation}\label{5.8}
		\frac{p^{u}\xi_{1}c_{1}+c_{1}c_{2}}{t_{1}c_{2}p^{u}}+\frac{p^{u}\xi_{2}-c_{1}}{t_{2}p^{u}}- \frac{\ell_1 }{t_{2}}
		=\frac{i_{1,\ell_{1}}a_{2}}{p}+k_{2}^{({1,\ell_{1}})}.
	\end{equation}
	Note that \( \ell_1 \in \{0, 1, \cdots, p-1\} \) and \( i_{1, \ell_1} \in \{1, \cdots, p-1\} \). Then, we can select two distinct integers \( \ell_1', \ell_1'' \in \{0, 1, \cdots, p-1\} \) such that \( i_{1, \ell_1'} = i_{1, \ell_1''} \). From \eqref{5.8},
	\[
	\frac{\ell_1' - \ell_1''}{t_2} = k_2^{(1, \ell_1')} - k_2^{(1, \ell_1'')} \in \mathbb{Z},
	\]
	which leads to a contradiction, thus completing the proof.
\end{proof}

\section{Examples}

In the final section, we present some examples to emphasize our key findings.
Firstly, we give a counterexample to refute \cite[Conjecture 4.2]{CWZ}.
\begin{exam}\label{ex6.1}
Define
\[
 M = \begin{bmatrix}
	6 & \frac{3}{4} \\
	0 & 6
\end{bmatrix},
D = \left\{
\begin{pmatrix} 0 \\ 0 \end{pmatrix},
\begin{pmatrix} 1 \\ 0 \end{pmatrix},
\begin{pmatrix} 0 \\ 1 \end{pmatrix}
\right\}
\quad \text{and} \quad
P = \begin{bmatrix}
	8 & 0 \\
	0 & 1
\end{bmatrix}.
\]
Then
\[
M_1 =PMP^{-1}= \begin{bmatrix}
	6 & 6 \\
	0 & 6
\end{bmatrix}
\text{and} \quad
D_1 =PD= \left\{
\begin{pmatrix} 0 \\ 0 \end{pmatrix},
\begin{pmatrix} 8 \\ 0 \end{pmatrix},
\begin{pmatrix} 0 \\ 1 \end{pmatrix}
\right\}.
\]
It is known that \( \mu_{M,D} \) and \( \mu_{M_{1}, D_{1}} \) have the same spectrality. Let
\[
L_1 = \left\{ \begin{pmatrix} 0 \\ 0 \end{pmatrix}, \begin{pmatrix} 2 \\ 34 \end{pmatrix}, 2\begin{pmatrix} 2 \\ 34 \end{pmatrix}\right\}.
\]
We can check $(M_1,D_1,L_1)$ forms a Hadamard triple. Thus \( \mu_{M,D} \) and \(\mu_{M_1,D_1}\) are both spectral measure. The spectral measure $\mu_{M, D}$ can be verified to satisfy all conditions of Theorem \ref{thm1.5}, while $c=\frac{3}{4}=\frac{1}{8}\times 6=\kappa \beta$ and $p\kappa=2\times \frac{1}{8}\notin \mathbb Z$. This implies that \cite[Conjecture 4.2]{CWZ} is false. As shown in Figs. \ref{Fractal-1} to \ref{Fractal-4}, these figures also illustrate the first four approximations of the Sierpinski fractal of \( \mu_{M, D} \).
\begin{figure}[ht]
	\centering
	\begin{minipage}{0.24\textwidth}
		\centering
		\includegraphics[width=\textwidth]{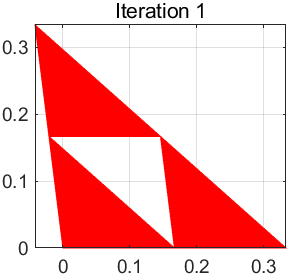}
		\caption{}
		\label{Fractal-1}
	\end{minipage}
	\hfill 
	\begin{minipage}{0.24\textwidth}
		\centering
		\includegraphics[width=\textwidth]{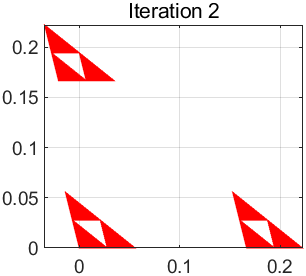}
		\caption{}
		\label{Fractal-2}
	\end{minipage}
		\hfill 
	\begin{minipage}{0.24\textwidth}
		\centering
		\includegraphics[width=\textwidth]{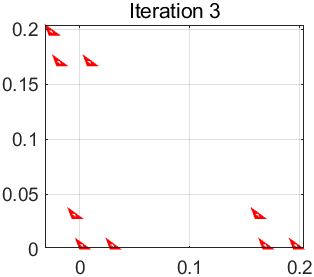}
		\caption{}
		\label{Fractal-3}
	\end{minipage}
	\hfill 
	\begin{minipage}{0.24\textwidth}
		\centering
		\includegraphics[width=\textwidth]{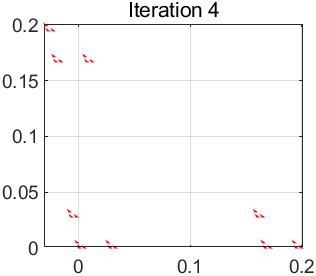}
		\caption{}
		\label{Fractal-4}
	\end{minipage}
\end{figure}
\end{exam}

\begin{exam}
	Let
	\begin{equation*}
		M=\begin{bmatrix} \rho_1^{-1} & c \\ 0 & \rho_2^{-1} \end{bmatrix}\subset M_{2}(\mathbb{R}),
\quad \text{and} \quad
D= \left\{
		\begin{pmatrix} 0 \\ 0 \end{pmatrix},
		\begin{pmatrix} 1 \\ 0 \end{pmatrix},
		\begin{pmatrix} 1 \\ 1 \end{pmatrix},
		\begin{pmatrix} 2 \\ 1 \end{pmatrix},
		\begin{pmatrix} 2 \\ 2 \end{pmatrix}
		\right\}.
	\end{equation*}
Then
	\begin{align*}
		\mathcal{Z}(m_{D})&=\{(x_{1}, x_{2})^{*}| 1+e^{-2\pi ix_{1}}+e^{-2\pi i(x_{1}+x_{2})}+e^{-2\pi i(2x_{1}+x_{2})}+e^{-2\pi i(2x_{1}+2x_{2})}=0\}\\
&=\bigcup_{j=1}^{4}\left(\frac{j}{5}\begin{pmatrix}
			1\\
			1\\
		\end{pmatrix}+\mathbb{Z}^{2}\right).
	\end{align*}
According to Theorem \ref{thm1.6}, the measure \( \mu_{M,D} \) is a spectral measure if and only if the matrix \( M \) satisfies one of the following conditions:
	\begin{enumerate}
		\item[$(i)$] \( \rho_1 \neq \rho_2 \), \( \rho_1^{-1} \in 5\mathbb{Z} \), and \( c'' \in \mathbb{Q}\setminus E_{\bm{a}} \), where $E_{\bm{a}}=E_{(1,1)^*}=\{\frac{s}{t}:t-s\in\mathbb{Z}\backslash5\mathbb{Z},\;\gcd(s,t)=1\}\cup\{0\} $.
		\item[$(ii)$] \( \rho_1\neq \rho_2 \), \( \rho_1^{-1}, \rho_2^{-1} \in 5\mathbb{Z} \), and \( c'' = \frac{c_{1}}{c_{2}} \in E_{\bm{a}} \), \( c_{2} \in 5^{\ell}(\mathbb{Z} \setminus 5\mathbb{Z}) \), \( 5^{\ell+1} \mid (\rho_1^{-1} - \rho_2^{-1}) \) for some \( \ell \in \mathbb{N} \).
		\item[$(iii)$] \( \rho_1 = \rho_2 \in 5\mathbb{Z} \) and \( c \in \{ \frac{s}{t} : s \in 5\mathbb{Z}, \gcd(s, t) = 1 \} \cup \{0\} \).
	\end{enumerate}
\end{exam}

We present the following example to illustrate Theorems \ref{thm1.7} and \ref{thm1.8}.
\begin{exam}\label{exam3}
	Let
	\[
	D=\left\{
	\begin{pmatrix} 0 \\ 0 \end{pmatrix},
	\begin{pmatrix} 1 \\ 0 \end{pmatrix},
	\begin{pmatrix} 1 \\ -1 \end{pmatrix},
	\begin{pmatrix} 2 \\ -1 \end{pmatrix},
	\begin{pmatrix} 2 \\ -2 \end{pmatrix}
	\right\},
	\]
	and
	\[
	M_1 = \begin{bmatrix}
		\sqrt{5} & \frac{\sqrt{5}}{3} \\
		0 & \sqrt{5}
	\end{bmatrix},\;
	M_2 = \begin{bmatrix}
		\frac{3\sqrt{5}}{5} & \frac{\sqrt{5}}{25}   \\
		0 & \frac{3\sqrt{5}}{5}
	\end{bmatrix},\;
	M_3 = \begin{bmatrix}
		\frac{3\sqrt{5}}{5} & 1   \\
		0 & \frac{3\sqrt{5}}{5}
	\end{bmatrix},\;
	M_4 = \begin{bmatrix}
		\sqrt{7}  & \sqrt{7}\\
		0 & \sqrt{7}
	\end{bmatrix}.
	\]
Then by Theorems \ref{thm1.7} and \ref{thm1.8}, \( L^2(\mu_{M_1,D}) \) admits an infinite set of orthogonal exponential functions, while \( L^2(\mu_{M_2,D}) \) contains an arbitrary number of orthogonal exponential functions. Moreover, there can be at most $5$ mutually orthogonal exponential functions in both \( L^2(\mu_{M_3,D}) \) and \( L^2(\mu_{M_4,D}) \), with $5$ being the maximum possible. Figs. \ref{Fractal-5} to \ref{Fractal-8} show the second iterative fractal graphs of the measure \( \mu_{M_1, D} \), \( \mu_{M_2, D} \), \( \mu_{M_3, D} \), and \( \mu_{M_4, D} \), respectively.
\begin{figure}[ht]
	\centering
	\begin{minipage}{0.24\textwidth}
		\centering
\includegraphics[width=\textwidth]{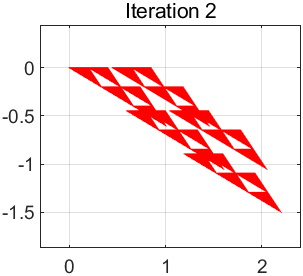}
		\caption{}
		\label{Fractal-5}
	\end{minipage}
	\hfill 
	\begin{minipage}{0.24\textwidth}
		\centering
\includegraphics[width=\textwidth]{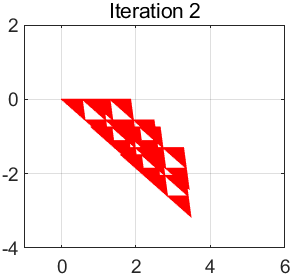}
		\caption{}
		\label{Fractal-6}
	\end{minipage}
	\begin{minipage}{0.24\textwidth}
		\centering
\includegraphics[width=\textwidth]{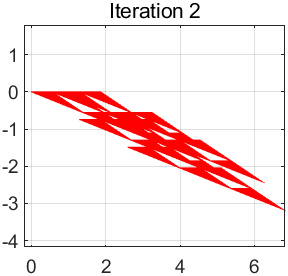}
		\caption{}
		\label{Fractal-7}
	\end{minipage}
	\hfill 
	\begin{minipage}{0.24\textwidth}
		\centering
		\includegraphics[width=\textwidth]{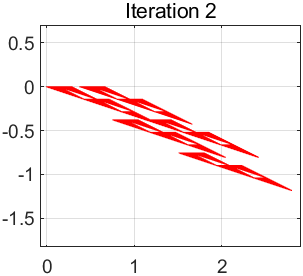}
		\caption{}
		\label{Fractal-8}
	\end{minipage}
\end{figure}
\end{exam}

Finally, we propose two natural questions.

\begin{que}
In Theorem \ref{thm1.6}, if we replace \(d_1 = (d_{1,1}, 0)^*\) with \(d_1 = (d_{1,1}, d_{1,2})^*, d_{1,2}\neq \bm{0}\) in the condition \eqref{eq1.3}, does the conclusion still hold?
\end{que}

\begin{que}
How to extend our results to higher dimensions or Moran measure?
\end{que}

\section*{Acknowledgment}

\textbf{Competing interests.} The authors declare that they have no competing interests.

\textbf{Data availability.}
Data availability is not applicable to this article as no new data were created or analyzed in this study.

\end{document}